\documentclass[11pt]{amsart}
\usepackage{graphicx}
\usepackage{amssymb}
\usepackage{epstopdf}
\usepackage{amsmath}
\usepackage{amscd}
\usepackage{amsthm,amsfonts,amsxtra,amssymb,amscd}
\usepackage[all]{xy}
\usepackage{enumerate}

\def\pd#1#2{\frac{\partial#1}{\partial#2}}

\usepackage{hyperref}

\theoremstyle{plain}
\newtheorem*{lemma*}{Lemma}
\newtheorem{lemma}[subsection]{Lemma}
\newtheorem*{theorem*}{Theorem}
\newtheorem{theorem}[subsection]{Theorem}
\newtheorem*{proposition*}{Proposition}
\newtheorem{proposition}[subsection]{Proposition}
\newtheorem*{corollary*}{Corollary}
\newtheorem{corollary}[subsection]{Corollary}
\theoremstyle{definition}
\newtheorem*{definition*}{Definition}
\newtheorem{definition}[subsection]{Definition}
\newtheorem*{example*}{Example}
\newtheorem{example}[subsection]{Example}
\theoremstyle{remark}
\newtheorem*{remark*}{Remark}
\newtheorem{remark}[subsection]{Remark}

 \def\infdex{\text {infdex}}

\def\infdex {\text {{\rm infdex}}}

\newcommand{\R}{{\mathbb R}}
\newcommand{\Z}{{\mathbb Z}}

\newcommand{\g}{\mathfrak{g}}
\renewcommand{\l}{\mathfrak{l}}

\title[  The infinitesimal Index ]{The infinitesimal Index}
\author{
C. De Concini,\quad C. Procesi,\quad M. Vergne}
 \thanks{The first two authors are partially supported by the Cofin 40 \%, MIUR}

\begin{document}

\begin{abstract}

In this note, we study an  invariant associated to the zeroes
of the moment map  generated by an action form, the {\em
infinitesimal index}.
This construction  will be used  to study the compactly supported equivariant cohomology of the zeroes of the moment map and to give formulas for
 the multiplicity index map of a transversally elliptic operator.
   \end{abstract}\maketitle

\section*{Introduction}
Let $G$ be a compact Lie group acting on a manifold $N$. Then $G$
acts on the cotangent bundle $M=T^*N$ in a Hamiltonian way. Set $\mathfrak g:=T_1G$ its Lie algebra. The
set $M^0$ of zeroes of the moment map $\mu: M\to \mathfrak g^*$ is
the union of the conormals to the $G$-orbits in $N$. An  element
$S$ of the equivariant $K$ theory $K_G(M^0)$ of $M^0$  is called a
transversally elliptic symbol, and Atiyah-Singer (see \cite{At})
associated to $S$ a trace class representation ${\rm index}(S)$ of
$G$. If $\hat G$ is the dual of $G$, the representation ${\rm
index}(S)$ gives rise to a function $m(\tau)$ on $\hat G$: ${\rm
index}(S)=\sum_{\tau\in \hat G} m(\tau) \tau$ called the
multiplicity index map.

The analog of the equivariant $K$-theory of $M^0$ is the
equivariant cohomology with compact supports $H^*_{G,c}(M^0)$. Here
we construct a map $\infdex_G^{\mu}$, called the infinitesimal
index, associating to an element $[\alpha]\in H_{G,c}^*(M^0)$ an
invariant distribution on $\mathfrak g^*$. We prove a certain
number of functorial properties of this map, mimicking  the
properties of the  index map  formalized by Atiyah-Segal-Singer.
 However, although our proofs are similar to
 \cite{BV2}, \cite{BV1}, \cite{pep2},  \cite{par-ver2}, our point of view is dual.
  Indeed in  previous works, the  equivariant index, or integrals of equivariant cohomology classes,
  are (generalized) functions on $G$, or $\g$, while we work directly on the dual space $\g^*$.


 More generally, we consider the case when $M$ is a $G$-manifold provided with a $G$-invariant one
form $\sigma$ (and we do not assume that $d\sigma$ is non
degenerate). This allows us to obtain a map $\infdex_G^{\mu}: H^*_
{G,c}(M^0)\to \mathcal D'(\g^*)^G$, where $M^0$ is the set of
zeroes of the associated moment map $\mu: M\to \mathfrak g^*$ and
$\mathcal D'(\g^*)^G$ the space of $G$-invariant distributions on $\g^*$. Our construction is  strongly related to Paradan's localization  on $M^0$   of the equivariant cohomology of $M$ (see \cite{VergneICM}).

\bigskip

\subsection*{Outline of the article}

Let us summarize the content of this article.

In the first section, we give a  ``de Rham" definition of the equivariant
cohomology with compact supports $H_{G,c}^*(Z)$ of a topological
space $Z$ which is a closed invariant subspace of a $G$-manifold $M$: a representative of
a class $[\alpha]$ is an equivariant differential form $\alpha(x)$
on $M$ with compact supports and such that the equivariant
differential $D\alpha$ of $\alpha$ vanishes in a neighborhood of
$Z$.
In the appendix, we show under mild assumptions on $M$ and $Z$ that our space $H^*_{G,c}(Z)$ is naturally isomorphic with the (topological) equivariant    cohomology of $Z$ with compact supports.

In the second section, we define the infinitesimal index. Let $M$
be a $G$-manifold provided with a $G$-invariant one form $\sigma$
(we will say that $\sigma$ is an action form). Let $v_x$ be the
vector field on $M$ associated to $x\in \mathfrak g$ and $i_x$  the derivation on forms induced by  contraction with $v_x$. The moment
map $\mu:   \mathfrak g\to C^\infty (M)$ or  $\mu: M\to \mathfrak g^*$ is defined by
$\mu(x)=-\langle\sigma, v_x\rangle=-i_x(\sigma)$. Then
$$\Omega(x)=\mu(x)+d\sigma=D\sigma(x)$$ is a closed (in fact exact)
 equivariant form on $M$.  The symbol $D$ denotes in this paper the equivariant differential as  defined in the Cartan model (see Formula \eqref{eqdif}.)

Our main remark is the following
\begin{proposition}
If $f$ is a smooth function on $\mathfrak
g^*$ with compact support, and $\alpha$ is a compactly supported equivariant form such that the
differential $D\alpha$   vanishes in a neighborhood of
$M^0:=\mu^{-1}(0),$ then the double integral
$$\int_M\int_{\mathfrak g} e^{is\Omega(x)}\alpha(x) \hat f(x) dx$$
is independent of $s,$ for $s$ sufficiently large.
\end{proposition}  Here $\hat f(x)$
is the Fourier transform of $f$. Some comment is in order: if
$\alpha(x)$ is closed (and compactly supported) on $M$, it is
clear that the integral $\int_{M}e^{i s\Omega(x)}\alpha(x)$ is
independent of $s$ as $\Omega(x)=D\sigma(x)$ is an exact
equivariant form. In our context, $\alpha(x)$ is compactly
supported, but $\alpha(x)$ is {\bf not closed} on $M$: only its
restriction to a neighborhood of $M^0$ is closed. This is however
sufficient to prove that
\begin{equation}\label{definf}
 \langle \infdex_G^{\mu}([\alpha]),f\rangle = \lim_{s\to \infty} \int_M
 \int_{\mathfrak g}
 e^{is \Omega(x)} \alpha(x) \hat f(x) dx
 \end{equation}
is a well defined map from $H_{G,c}^*(M^0)$ to invariant
distributions on $\mathfrak g^*$. This we call the {\em infinitesimal index}.
The infinitesimal index  does not depend of some deformations of the form $\sigma$, see Theorem \ref{dein}.

In the third and fourth sections, we prove a certain number of functorial
properties of the infinitesimal index: the locality (excision) property in Subsection \ref{excision}, the functoriality with respect to subgroups in Subsection \ref{subgroup},
the stability with respect to immersions in Subsection \ref{push}.

One of the most important properties is
the free action property that we prove in Subsection \ref{sub:fre}.
Consider the situation where the compact
Lie group $L$ acts freely on $M$ and $0$ is a regular value of
$\mu$. Then the infinitesimal index of a class $[\alpha]$ is a
polynomial density on $\mathfrak l^*$. Its value at $0$ is the
integral of the cohomology class corresponding to $[\alpha]$ by
the Kirwan map over the reduced space $\mu^{-1}(0)/L$. This is
essentially Witten nonabelian localization theorem \cite{wit}.
 We  give also the double equivariant version, where  a compact Lie group $G$ acts
on $M$ commuting with the free action of $L$.

We then deduce from these properties the stability with respect to induction  in Subsection \ref{inductioninfdex}, and a comparison formula with the infinitesimal index  for the maximal torus of $G$  in Subsection \ref{maxtori}.

\bigskip

Let us comment on previous work around this theme.

The use of the form $e^{i s D\sigma}$, in order  to "localize"
integrals, is the main principle in Witten nonabelian localization
theorem \cite{wit}, \cite{jef-kir}, and our definition of the
infinitesimal index
 is strongly inspired by this principle.

P.-E. Paradan has studied systematically the situation of a
manifold $M$ provided with a $G$-invariant action form $\sigma$.
Indeed, he constructed in \cite{pep2} a closed equivariant form
$P_\sigma$ on $M$, congruent to $1$ in cohomology and
 supported near $M^0$.
 Paradan's form $P_\sigma$ is constructed
using equivariant cohomology with $C^{-\infty}$-coefficients.
Multiplying $\alpha(x)$ by Paradan's form $P_\sigma(x)$ leads to a
closed compactly supported equivariant form on $M$ and $I(x):=\int_M P_\sigma(x)\alpha(x)$ is a generalized function
on $\mathfrak g$. As we explain in Remark \ref{par}, our
infinitesimal index is the Fourier transform of $I(x)$. Properties
of the infinitesimal index could thus be deduced by Fourier
transform from the functorial properties of $P_\sigma$ proven in
\cite{pep2}, \cite{par-ver2}.
For example, the independence of the infinitesimal index  with respect to some deformations of the form $\sigma$ is an important tool,
similar to independence  for $P_\sigma$ proven in more general setting in \cite{pep2}(Proposition 2.6).
 The formula for $\infdex_G^{\mu}$ in function of a maximal torus of $G$ is similar to a remarkable formula  in  \cite{pep2} (Theorem 4.5).
However, we have chosen here to prove directly   properties of the infinitesimal index
  by using  our  limit definition.
There are two advantages in doing so. First,  we believe that
the proofs are easier.
Secondly, this is the framework
we will use in  \cite{dpv3} to produce piecewise polynomial densities (also called spline distributions)
directly on  $\mathfrak g^*$ from some compactly supported equivariant classes.

In the case where $M=N\oplus N^*$, where $N$ is a representation
space for a linear action of a torus $G$,   we determined $K_G(M^0)$ as a space of
functions on $\hat G$ using the multiplicity index map  in the article
\cite{dpv1}.   In a companion article \cite{dpv2},    we have used the
infinitesimal index to identify $H_{G,c}^*(M^0)$ to a space of
spline distributions on $\mathfrak \g^*$ of which the functions describing the index are a discrete analogue.

Finally using the analogies between splines and discrete functions, we  have compared in \cite{dpv3}
 the equivariant cohomology with compact supports and the equivariant $K$-theory of $M^0$, by relating the infinitesimal index and the multiplicity index map.

 We wish to thank  Paul-\'Emile Paradan for his comments.

\section {Equivariant de Rham cohomology\label{EdR}}

Let  $M$ be a $C^\infty$ manifold with a  $C^\infty$ action of a compact Lie group $G$. We are going to define its equivariant cohomology with compact supports
 following  Cartan definition  (see \cite{gui-ste99}).

We define   the space of compactly supported equivariant forms as
$${\mathcal A}_{G,c}(M)=(S(\mathfrak g^*)\otimes {\mathcal A}_{c} (M))^G$$
with the grading given  setting $\mathfrak g^*$ in degree 2. Here
$ {\mathcal A}_{c} (M)$ is the algebra of differential forms on
$M$ with compact supports.
Thus an element of ${\mathcal A}_{G,c}(M)$ can be written
$\alpha(x) =\sum_{a=1}^R P_a(x) \alpha^a$ where $P_a(x)$ are  polynomial functions on  $\mathfrak g$, and $\alpha^a$   differential forms with compact support on $M$.

Each element $x\in\mathfrak g$  of the Lie algebra of $G$ induces
a vector field $v_x$ on $M$, the {\em infinitesimal generator} of
the action: here the sign convention is that
$v_x=\frac{d}{d\epsilon} \exp (-\epsilon x)\cdot m$ in order that
the map $x\to v_x$ be a Lie algebra homomorphism. A vector field
$V$ on $M$  induces a derivation $\iota_V$ on forms, such that
$\iota_V(df)=V(f)$, and for simplicity we denote by
$\iota_x=\iota_{v_x}$.

One defines the differential as follows. Given $\alpha\in
{\mathcal A}_{G,c}(M)$, we think of $\alpha$ as an equivariant
polynomial map on $\mathfrak g$ with values in $ \mathcal  A_c(M)$,
thus, for any $x\in\mathfrak g$, we set
\begin{equation}
\label{eqdif}D\alpha(x):=d(\alpha(x))-\iota_x (\alpha (x))
\end{equation}
where $d$ is the usual de Rham differential.

It is easy to see that $D$  increases the degree by one and that
$D^2=0$. Thus we can take cohomology and we get the
$G$-equivariant cohomology of $M$ with compact supports.

Now take a $G$-stable closed set $Z$ in a manifold $M$.
 Consider the open set $U=M\setminus Z$.
 Then $U$ is a manifold and we have an inclusion of complexes ${\mathcal A}_{G,c}(U)\subset {\mathcal A}_{G,c}(M)$ given by extension by zero.
    We set $${\mathcal A}_{G,c}(Z,M):={\mathcal A}_{G,c}(M)/{\mathcal A}_{G,
    c}(U).$$

 \begin{definition}\label{due} The  equivariant de Rham cohomology  with compact supports $H^*_{G,c}(Z)$ is  the cohomology of the complex ${\mathcal A}_{G,c}^*(Z,M)$.
\end{definition}

  Notice that ${\mathcal A}_{G,c}(U)$ is an
ideal in $ {\mathcal A}_{G,c}(M)$ so ${\mathcal A}_{G,c}(Z,M)$ is a
differential graded algebra and $H^*_{G,c}(Z)$ is a graded algebra
(without 1 if $Z$ is not compact).

In this model, a representative of a class in $H^*_{G,c}(Z)$ is an
equivariant form $\alpha(x)$ with {\bf compact support} on $M$.
The form $\alpha$ is not necessary equivariantly closed on $M$,
but there exists a neighborhood  of $Z$ such that the restriction
of $\alpha(x)$ to  this neighborhood is equivariantly closed.

If $Z$ is compact, the class $1$ belongs to $H^*_{G,c}(Z)$: a
representative of $1$ is a $G$-invariant function $\chi$ on $M$
with compact support and identically equal to $1$ on a
neighborhood of $Z$ in $M$.

\begin{remark}
 Our model for $H^*_{G,c}(Z)$ seems to depend of the ambient manifold $M$. However, in the appendix we are going to see that, under mild assumptions on $M$ and $Z$, $H^*_{G,c}(Z)$ is naturally isomorphic with the equivariant   singular cohomology of $Z$ with compact supports.

 \end{remark}

By the very definition of $H^*_{G,c}(Z)$, we also deduce

\begin{proposition}\label{numero6}
Let $M$ be a $G$-space, $j:Z\to M$ the inclusion of a closed $G$-stable subset. Denote
by $i:U\to M$ the inclusion  of $U:=M\setminus Z$. We have a long exact sequence
\begin{equation}\label{lunga}
\begin{CD}\cdots\to H^h_{G,c}(U)@>i_*>> H^h_{G,c}(M)@>j^*>> H^h_{G,c}(Z)@> >>  H^{h+1}_{G,c}(U)\to\cdots .\end{CD}
\end{equation}
\end{proposition}

\bigskip

If $i:Z\to M$ is a closed  $G$-submanifold of  a manifold $M$,
 the restriction of forms gives rise to a well defined map
 $i^*:{\mathcal A}_{G,c}(Z,M)\to {\mathcal A}_{G,c}(Z)$.

 \begin{proposition}
If $Z$ is a closed $G$-invariant  submanifold of  a manifold  $M$ admitting an
equivariant tubular neighborhood, the  map $i^*$  induces an isomorphism in
cohomology.
\end{proposition}
\begin{proof}
  We  reduce to the case in which $M$
  is a vector bundle on $Z$ by restriction to a tubular neighborhood.
 Put a $G$-invariant metric on this bundle and
 let $p:M\to Z$ be the projection.
 Choose a  $C^\infty$ function $f$ on $\mathbb R$ with compact support
 and equal to 1 in a neighborhood of 0. We map an equivariant  form
 $\omega\in\mathcal  {\mathcal A}_{G,c}(Z)$ to
 $f(\|m\|^2)p^*\omega(m)$
 and then to its class modulo ${\mathcal A}_{G,c}(U)$.
 It is easily seen that this map is an inverse in cohomology of the map $i^*$.
\end{proof}

 Assume that $M$ is a $L\times G$ manifold and that  $L$ acts freely on $M$ so that $M/L$ is a manifold with a   $G$--action. Let $Z$ be a $G\times L$ closed subset of $M$.
   Denote by $p:M\to M/L$ the projection. The pull back of forms on $M/L$ induces a map from $p^*:H^*_{G,c}(Z/L)\to H^*_{L \times G,c}(Z)$.

  \begin{proposition}\label{pro:free}  The pull back  $$p^*:H^*_{G,c}(Z/L)\to H^*_{L \times G,c}(Z)$$   is an isomorphism
\end{proposition}
\begin{proof} The fact that  the pull back of forms induces an isomorphism between
$H^*_{G,c}(M/L)$ and $ H^*_{L \times G,c}(M)$,  and between $H^*_{G,c}((M\setminus Z)/L)$ and $ H^*_{L \times G,c}(M\setminus Z)$,
is proven as in Cartan  (see \cite{gui-ste99} or \cite{duf-kum-ver}).
Our statement then follows from Proposition \ref{numero6}.
\end{proof}

\section{Basic definitions}

  \subsection{ Action form and the moment map}

 Let $G$ be a Lie group and  $M$
a $G$-manifold.  \begin{definition}
An {\em action form} is a $G$-invariant  real one form
$\sigma$ on $M$.
\end{definition}

 The prime examples  of this setting are when  $M$ is even
dimensional and $d\sigma$ is nondegenerate. In this case
$d\sigma$ defines a symplectic structure on $M$.

\begin{example}\label{maes}
For  every manifold $N$, we may take
  its cotangent bundle
$M:=T^*N$ with projection $\pi: T^*N\to N$. The canonical action
form $\sigma$   on a tangent vector $v$ at a point $(n,\phi),\
n\in N,\phi\in T^*_nN$  is given by
$$\langle\sigma\,|\,v\rangle:=\langle\phi\,|\,d\pi(v)\rangle.$$
In this setting, $d\sigma$ is a canonical symplectic structure on $T^*N$ and, if $r=\dim(N)$,  the form $\frac{d\sigma^r}{r!}$  determines an orientation and a measure, the {\em Liouville measure} on $T^*N$. If  a group
$G$ acts on $N$, then it acts also on $T^*N$ preserving the
canonical action form and hence the symplectic structure and the Liouville measure.
\end{example}\smallskip

\begin{remark} If $M$ is a  manifold with a $G$--invariant Riemannian
structure, we can consider  an invariant vector field  instead of
a one form.
\end{remark}

 \begin{definition} Given an action form $\sigma$
we  define the {\em moment map} $\mu_{\sigma}:M\to\mathfrak g^*$ associated to $\sigma$ by:
\begin{equation}\label{moma}
\mu_{\sigma}(m)(x):=-\langle \sigma\,|\,v_x\rangle(m)=-\iota_x (\sigma)(m)
\end{equation}
for $m\in M$, $x\in \mathfrak g$.
\end{definition}

\begin{remark}

Due to our sign convention for $v_x$, we have

$$\mu(m)(x):=\langle \sigma\,|\,\frac{d}{d\epsilon}\exp(\epsilon x)\cdot
m\rangle.$$

 The moment map is a
$G$--equivariant map, where on $\mathfrak g^*$ we have the {\em
coadjoint action}.

\end{remark}

The form  $d\sigma$ is  a closed  two form on $M$. Then
$D\sigma(x)=\mu(x)+d\sigma$ is a closed (in fact exact)
equivariant form on $M$.

  Let us present a few examples.

\begin{example}
 \label{ts} In Example \ref{maes},
   take  $N=S^1=\{e^{i\theta}\}$. The form $d\theta$ gives a trivialization    $T^*S^1=S^1\times \R$. The vector field $\pd{}{\theta}$   gives a  canonical generator  of the Lie algebra of $S^1$ and $d\theta$,  a generator  for the dual. The circle group $S^1$ acts
freely by rotations on itself. If $[e^{i\theta},t]$ is a
point of $T^*S^1$ with $t\in \R$, the action form $\sigma $  is
$\sigma=t d\theta$. The function $t$ is the moment map and
$dt\wedge d\theta$ the symplectic form.

 More generally, take $N=G$   a  Lie group.  Denote by
 $$L(g):h\mapsto gh,\ R(g):h\mapsto  hg^{-1}$$ the left and right actions of $G$ on $G$ and by extension also on $T^*G$.
 Let us now trivialize  $T^*G=G\times  \mathfrak g^*$ using left invariant  forms.
 Then, in this trivialization, for $h\in G$ and $\xi\in \mathfrak g^*,$
 $$L(g)(h,\xi)=(gh,\xi),\hspace {1cm} R(g)(h,\xi)=(hg^{-1},g\xi).$$


Call $\pi: T^*G\to G$ the canonical projection.  Fix a basis $\psi_1,\ldots,\psi_r$ of left invariant one forms on $G$ so that a point of   $T^*G=G\times  \mathfrak g^*$ is a pair $(g,\zeta)=(g,\sum_i\zeta_i\psi_i)$.  Clearly the action form is  $\sigma=\sum_i\zeta_i\pi^*(\psi_i)$, the symplectic form is  $\sum_id\zeta_i\wedge \pi^*(\psi_i)+\sum_i\zeta_i\pi^*(d\psi_i)$. In the noncommutative case,  in general $d\psi_i\neq 0$, nevertheless when we compute the Liouville form we immediately see that these terms disappear and
\begin{equation}
\label{liofo}  \frac{d\sigma^r}{r!}= d\zeta_1\wedge \pi^*(\psi_1)\wedge\cdots\wedge d\zeta_r\wedge \pi^*(\psi_r).
\end{equation}
We can rewrite this as $\frac{d\sigma^r}{r!}= (-1)^{\frac{r(r-1)}{2}}  d\zeta \wedge  \pi^*(V_{\psi})$
    where we set
 $V_{\psi}:=\psi_1\wedge \psi_2\wedge \cdots\wedge  \psi_r$ and  $d\zeta:=d\zeta_1\wedge\cdots \wedge d\zeta_r$.
 At this point it is clear that $V_{\psi}$  gives   a Haar measure $dg$ on $G$ while $d\zeta$ gives   a translation invariant measure on $\mathfrak g^*$.
 Thus we have (identifying top forms with measures according to the orientation of $T^*G$ given by $\frac{d\sigma^r}{r!}$)
 \begin{equation}
\label{liofo1}  \frac{d\sigma^r}{r!}= d\zeta dg.
\end{equation}

 \begin{remark}
We can further normalize our choice of the basis  of left invariant forms so that  $V_{\omega}$  gives the normalized  Haar measure   giving volume 1 to $G$. This normalizes also the  translation invariant measure on $\mathfrak g^*$.  The only further choice consists in choosing a  orientation for $G$ so that we have an induced orientation for $\mathfrak g^*$  giving the canonical orientation on $T^*G$.
\end{remark}  \smallskip

  Let us call  $\mu_\ell,\mu_r$  the moment maps for the left or right action of
  $G$ respectively.
By the very definition of $\sigma$, we have
\begin{proposition}
\begin{equation}\label{mmm} \mu_\ell(g,\psi)=  g\psi,\quad \mu_r
(g,\psi)= -\psi,  \quad\text{left trivialization}.
\end{equation}
\end{proposition}

%

If we had used right invariant forms in order to trivialize the
bundle, we would have
\begin{equation}
\label{mmm1} \mu_r(g,\psi)=-g\psi,\quad \mu_\ell
(g,\psi)=  \psi, \quad\text{\em right trivialization}.
\end{equation}
\end{example}

\begin{example}\label{v}
We get another example in the case of a symplectic vector space
$V$ with   antisymmetric form $B$. Then $\sigma=\frac{1}{2}B(v,dv)$
is a one  form on $V$ invariant under the action of the symplectic
group $G$ so that $d\sigma=\frac{1}{2}B(dv,dv)$ is a symplectic
two form on $V$. The moment map $\mu:V\to \mathfrak g^*$ is given
by $\mu(v)(x)=\frac{1}{2}B( v,xv)$.

For example,  let $M:=\R^2$ with coordinates $v:=[v_1,v_2],\
B=\begin{vmatrix} 0&1\\-1&0\end{vmatrix}$. The action form
$\sigma$ is $\frac{1}{2} (v_1 dv_2-v_2 dv_1)$ and
$d\sigma=dv_1\wedge dv_2$. The compact part of the symplectic
group is the circle group $S^1$ acting by rotations. The moment
map  is $\frac{v_1^2+v_2^2}{2}.$
\end{example}

\begin{remark}  Given  a vector space $N$,
  the space $N\oplus N^*$ has a canonical symplectic structure given by
\begin{equation}\label{casp}
\langle (u,\phi)\,|\, (v,\psi)\rangle:=\langle  \phi \,|\,  v
\rangle-\langle  \psi \,|\, u\rangle.
\end{equation}
The symplectic structure $d\sigma$  on the cotangent bundle to a
vector space $N$ gives a symplectic structure $B$ to the vector
space $T^*N=N\oplus N^*.$

 The action form   $\sigma$ coming from  the cotangent structure
 is not the same than the action form on $N\oplus N^*$ given by duality \eqref{casp}
 (in case $V=\R$, $ydx$ versus $\frac{1}{2}(ydx-xdy)$),
 but the moment map relative to the subgroup $GL(N)$ acting by $(gn,^tg^{-1}\phi)$
 is the same, as well as $d\sigma$.
\end{remark}
\subsection{The cohomology groups $\mathcal
H^{\infty}_{G,c}(M)$\label{infcc}}
We will need to extend the  notion of equivariant
cohomology groups. Consider the space $C^\infty(\mathfrak
g)$ of $C^{\infty}$ functions on $\mathfrak g$. We may consider the $\Z/2\Z$-graded spaces ${\mathcal A}^{\infty}_{G}(M
)$  (or ${\mathcal A}^{\infty}_{G,c}(M
)$  consisting of the $G$-equivariant $C^\infty$ maps  from $\mathfrak g$ to ${\mathcal A}(M)$ (or to ${\mathcal A}_c(M)$). The
equivariant differential $D$ is well defined on ${\mathcal
A}^{\infty}_{G}(M )$ ( or  on ${\mathcal
A}^{\infty}_{G,c}(M )$) and takes even forms to odd forms and vice
versa. Thus we get the cohomology groups $\mathcal
H^{\infty}_{G}(M)$, $\mathcal
H^{\infty}_{G,c}(M)$. The group $\mathcal H^{\infty}_{G,c}(M)$ is a module over $\mathcal H^{\infty}_{G}(M)$, and in particular on $C^\infty
(\mathfrak g)^G=\mathcal H^{\infty}_{G}(pt)$.

 Proceeding
as in the previous case,  we may define    for
any $G$-stable closed subspace $Z$ of $M$ the cohomology groups $\mathcal H^{\infty}_{G,c}(Z)$. An element in $\mathcal H^{\infty}_{G,c}(Z)$ is thus
represented by an element in ${\mathcal A}^{\infty}_{G,c}(M )$
whose boundary has support in $M\setminus Z$.
We have a natural map $H^*_{G,c}(Z)\to  \mathcal
H^{\infty}_{G,c}(Z)$.

In order to take Fourier transforms, we will need to use yet another space.

Consider the space $\mathcal P^\infty(\mathfrak
g)$ of $C^{\infty}$ functions on $\mathfrak g$ with at most polynomial growth. Equivalently, we  say  that $\mathcal P^\infty(\mathfrak
g)$ consists of functions with moderate growth.
 We may consider the spaces ${\mathcal A}^{\infty,m}_{G,c}(M
)$ consisting of the $G$-equivariant $C^\infty$ maps  with
at most polynomial growth from $\mathfrak g$ to ${\mathcal A}_c(M)$.
The index $m$ indicates the moderate growth  on $\mathfrak g$ of the coefficients.
We  get the cohomology groups $\mathcal
H^{\infty,m}_{G,c}(M)$. This new cohomology has $\mathcal
H^{\infty,m}_{G}(pt)= \mathcal P^\infty (\mathfrak g)^G$ and is a module over  $\mathcal P^\infty (\mathfrak g)^G$.
 We may define in the same way the  groups  $\mathcal H^{\infty,m}_{G,c}(Z)$ of cohomology with compact supports, and with coefficients of at most polynomial growth, for
any $G$-stable closed subspace $Z$ of $M$.

\subsection{Connection forms}
We shall use a fundamental notion in Cartan's theory of equivariant cohomology. Let us recall
\begin{definition}
Given a free action of a compact Lie group $L$ on a manifold $P$, a {\em connection form} is a $L$-invariant one  form  $\omega\in \mathcal A^1(P)\otimes \mathfrak l$  with coefficients in the Lie algebra of $L$ such that $-\iota_x \omega=x$ for all $x\in  \mathfrak l$.

\end{definition}
If on $P$ with free $L$ action we also have a commuting action of another compact group $G$,  it is easy to see that there exists
     a $G\times
L$ invariant connection form $\omega\in \mathcal A^1(P)\otimes \mathfrak l$  on $P$ for the free action of $L$.

Let $M=P/L$  and $y\in \mathfrak g$.
Define
the curvature $R $  and the $G$-equivariant curvature $R_y $ of the bundle
$P\to M$ by \begin{equation}
\label{ledc} R:=d\omega+\frac{1}{2}[\omega,\omega],\quad R_y:=-i_y\omega+R.
\end{equation}
\begin{example}\label{gecc}
Consider $L=G$ and $P=G$ with left and right action. A connection form for the right action can be constructed as follows.  Each element $x$ of the Lie algebra of $G$ defines the  vector field   $v_x$   by right action. These are left invariant vector fields.
Given a basis $e_1,\ldots,e_r$ of  $\mathfrak g$,  set $v_i:=v_{e_i}$.  This determines a dual basis and correspondingly left invariant forms $\omega_i$ with $i_{v_j}(\omega_i)=\langle \omega_i\,|\,v_{ j}\rangle=\delta^i_j$ so that $-\sum_i\omega_ie_i$ is a connection form for the right action.
%

This form is also left invariant and $R=0$,  so by \eqref{ledc}  the equivariant curvature is  $-i_y\omega$ where now  $i_y$ is associated to the left action. We then have
\begin{equation}
\label{cugu}R_y(g)=-\sum_ii_y(\omega_i)(g)e_i=-g^{-1}y.
\end{equation}
 \end{example}

 \bigskip

 The
equivariant Chern-Weil homomorphism  (\cite{ber-ver83},\cite{bot-tu}, see \cite{ber-get-ver})  associates to any
  $L$ invariant smooth function
  $a$ on $\l$ a closed $G$-equivariant form, with $C^{\infty}$-coefficients as in \S \ref{infcc}, denoted by $y\to a(R_y)$, on $M=P/L$.

  The  formula for this form is obtained via the Taylor series of the function $a$ as follows.  Choose a basis $e_j,\ j=1,\ldots,r$ of $\mathfrak l$ and write $R=\sum_j R_j\otimes e_j$. For a multi--index $I:=(i_1,\ldots,i_r)$, denote by $R^I:=\prod_{j=1}^r R_j^{i_j}$.  Then,  given a point $p\in P$, we set
 \begin{definition}\label{ary}
   \begin{equation}
\label{aryp}a(R_y)(p):=a(-i_y\omega+R)=a(-\iota_y\omega(p))+\sum_{I} {\partial_I}a (-\iota_y\omega(p))\frac{R^{I} }{I!}
\end{equation}
  which is  a finite sum since $R$ is a nilpotent element.
\end{definition} One easily verifies that this is independent of the chosen basis.
  Moreover one can prove (as in the construction of ordinary characteristic classes) the following proposition.
\begin{proposition}(\cite{ber-ver83},\cite{bot-tu}, see \cite{ber-get-ver})
 The differential form $a(R_y)$
  is the pull back of a $G$-equivariant closed form (still denoted by $a(R_y)$) on $M=P/L$.
  Its cohomology class in $\mathcal H^{\infty}_{G}(M)$ is independent of the choice of the connection.
\end{proposition}

\section{ Definition of the infinitesimal index}

 \subsection{Infinitesimal index}

As before,  consider  a compact Lie group $G$ and  a $G$-manifold $M$
equipped with an action form $\sigma$. We assume $M$
oriented. Let  $\mu:=\mu_\sigma:M\to  \mathfrak g^*$ be the
corresponding moment map  given by \eqref{moma}.

Set $$ M^0_G:=\mu^{-1}(0),\hspace{0.5cm} U:=M\setminus M^0_G.$$

We simply denote $M^0_G$ by $M^0$ when the group $G$ is fixed.

 Consider the equivariant form
$$\Omega:=d\sigma+\mu=D\sigma.$$

Let  $\mathcal D'( \mathfrak g^*)$ be the space of distributions
on $\mathfrak g^*$. It is a $S[\mathfrak g^*]$-module
 where $\mathfrak g^*$ acts as derivatives.
 When $G$ is noncommutative,
  we need to work with  the space  $\mathcal D'( \mathfrak g^*)^G$  of $G$-invariant distributions.

By Lemma \ref{due}, a representative of  a class $[\alpha] \in
H^*_{G,c}(M^0)$ is a form $\alpha \in [S( \mathfrak g^*)\otimes
\mathcal A_{c}(M)]^G$
  such that
$D\alpha $ is compactly supported in $U$.\smallskip

Let us define a map called the infinitesimal index
$${\rm infdex}_G^{\sigma}:H_{G,c}^*(M^0)\to \mathcal D'( \mathfrak g^*)^G$$
as follows.

We fix a  translation invariant  Lebesgue measure $d\xi$ on $\mathfrak g^*$. We choose a square root $i$ of $-1$ and define
the Fourier transform:
$$\hat f(x):= \int_{\mathfrak g^*}e^{-i\langle \xi\,|\,x\rangle}f(\xi)d\xi.$$
We normalize  $dx$  on $\mathfrak g$ so that  the inverse Fourier
transform is
\begin{equation}\label{inve}
f(\xi)=\int_{\mathfrak g}e^{i\langle \xi\,|\, x\rangle}\hat
f(x)dx.\end{equation} \vskip10pt The measure $dx d\xi$ is
independent of the choice of $d\xi$.

\medskip

Let $f(\xi)$ be a $C^\infty$ function on $ \mathfrak g^*$ with
compact support in  a ball $B_R$ of radius $R$ in $ \mathfrak g^*$ (for a choice of Euclidean structure on $\g^*$)
and $\hat f(x)$ its Fourier transform, a rapidly decreasing
function on $ \mathfrak g$.

Consider the differential form on $M$ depending on a parameter $s$:
$$\Psi(s,\alpha,f)=\int_{\mathfrak g}   e^{i s\Omega(x)}  \alpha (x) \hat f(x) dx,$$
and define
\begin{equation}
\label{infsas}\langle  {\rm infdex}(s,\alpha,\sigma), f\rangle :=
\int_{M}\int_{\mathfrak g} e^{i s\Omega(x )} \alpha(x )  \hat f(x
)  dx
\end{equation} $$= \int_M\Psi(s,\alpha,f) .$$ This double integral on $M\times
\mathfrak g$  is absolutely convergent, since $\alpha$ is
compactly supported on $M$ and depends polynomially on $x$ while
$\hat f(x )$ is rapidly decreasing.

More precisely, write $\alpha(x) =\sum_{a=1}^R P_a(x) \alpha^a$
  with
$\alpha^a$ compactly supported forms on $M$ and $P_a(x)$
polynomial functions of $x$. Then
$$\Psi(s,\alpha,f)(m)=
\sum_a\Big[ \int_{\mathfrak g} \hat f(x) P_a(x)e^{i s \langle
\mu(m),x\rangle } dx \Big] e^{i sd\sigma}\alpha^a .$$

By Fourier inversion (as in (\ref{inve}))
\begin{equation}
\label{treo}\int_{\mathfrak g}
\hat f(x) P_a(x)e^{i s \langle \mu(m),x\rangle } dx =
(P_a(-i\partial)  f)  (s\mu(m)),
\end{equation} thus

\begin{equation}\label{tre}\Psi(s,\alpha,f)= \sum_a((P_a(-i\partial)  f)\circ (s\mu) )
e^{i sd\sigma}\alpha^a.
\end{equation}

In particular, remark that  $\Psi(s,\alpha,f)$ does not depend of
the choice of $d\xi$. Another consequence of this analysis is
\begin{proposition}\label{ilsupd}
Let  $K\subset M$ be the support of $\alpha$ and $C\subset\mathfrak g^*$  the support of $f$.

The support of $\Psi(s,\alpha,f)  $ is contained in $K\cap \mu^{-1}(C/s)$. In particular,  if $s\mu(K)\cap C=\emptyset$, then $\Psi(s,\alpha,f)=0.$
\end{proposition}

\bigskip

 Given $s> 0$, set $V_s=\mu^{-1}(B_{R/s}).$
We can then  choose some $s_0>>0$ so large that the restriction of
$\alpha $  to the  small neighborhood  $ V_{s_0}$ of  $M^0$ is
equivariantly closed. This is possible since $D\alpha $ has a
compact  support $K$ in $U=M\setminus M^0$ so that, if $\rho:=\min_{m\in
K}\|\mu(m)\|>0$, it suffices to take $s_0>R/\rho$.

We have $(P_a(-i\partial)  f)  (s\mu(m))=0$  if $\|s\mu(m)\|>R\iff
\| \mu(m)\|>R/s.$ Thus we see that, for $s\geq s_0$, if $K$ is the
support of $\alpha$, $\Psi(s,\alpha,f)$  has compact support
contained in $V_s\cap K$.

We have then  the formula:
\begin{equation}
\label{intind} \langle  {\rm infdex}(s,\alpha,\sigma), f\rangle =
\int_{M} \Psi(s,\alpha,f)=\int_{V_s} \Psi(s,\alpha,f).
\end{equation}

Note that from Formula (\ref{intind})  follows the
\begin{lemma}\label{lem:support}
\label{sss} If $\alpha$ has support in $U$  then,  for $s$ large,
 $\Psi(s,\alpha,f)=0$.
\end{lemma}

  We will often make use of the following lemma.

\begin{lemma}\label{lem:dds}
We have
$$-i\frac{d}{ds}\int_M  \int_{\mathfrak g} e^{i s\Omega(x )}
\alpha(x ) \hat f(x)  dx=
 \int_M  \int_{\mathfrak g}   \sigma e^{i s\Omega(x )}
 D(\alpha)(x ) \hat f(x)  dx. $$
\end{lemma}
 \begin{proof}
Indeed, since $\Omega(x)=D\sigma(x)$,
$$-i\frac{d}{ds}\int_M  \int_{\mathfrak g} e^{i s\Omega(x )}
\alpha(x ) \hat f(x)  dx=
 \int_M \int_{\mathfrak g}  D\sigma(x)e^{i s\Omega(x)}
\alpha (x)   \hat f(x) dx$$
$$= \nu+r$$
with $$\nu=\int_{\mathfrak g} \int_M D\left(  \sigma e^{i
s\Omega(x )} \alpha(x ) \right)\hat f(x) dx$$
 and
$$r= \int_M  \int_{\mathfrak g}   \sigma e^{i s\Omega(x )} D( \alpha)(x ) \hat f(x)  dx
 $$
since $D(\Omega)=0$ and $D$ is a derivation, we have $D(e^{i
s\Omega(x)} )=0$.

As $\alpha(x)$ is compactly supported, $\nu=0$, and we obtain the
lemma.
\end{proof}

Let us see that

$$\langle  {\rm infdex}(s,\alpha,\sigma), f\rangle =
\int_M\int_{\mathfrak g} e^{i s\Omega(x )} \alpha(x )  \hat f(x )
dx $$ does not depend of the choice of  $s\geq s_0$.

We use Lemma \ref{lem:dds} above to compute $\frac{d}{ds}\langle
{\rm infdex}(s,\alpha,\sigma), f\rangle$. By the hypotheses made, the form $\sigma D\alpha$ has compact support in $U$, thus by Lemma
\ref{lem:support}   the differential form
$\Psi(s,\sigma D\alpha,f)=\int_{\mathfrak g} \sigma e^{i s\Omega(x
)} D( \alpha)(x ) \hat f(x)  dx$   is identically  equal to $0$
for $s\geq s_0$. This implies that for $s\geq s_0$
$$\frac{d}{ds}\langle  {\rm infdex}(s,\alpha,\sigma), f\rangle =0,$$
hence the independence of the choice of  $s\geq s_0$.

\bigskip

We now see the   independence on the choice of the representative
$\alpha $. In fact, take a different representative $\alpha+\beta$ with $\beta$  compactly supported on $U$.  Then
$$\lim_{s\to\infty} \langle \infdex(s,\beta,\sigma), f\rangle =0$$
 by Lemma \ref{lem:support}.

Next let us show that $\lim_{s\to\infty}\langle \infdex(s,\alpha,\sigma), f\rangle $ depends only on the cohomology class of $\alpha$.   Take  $\alpha=D\beta$,  with $\beta$ compactly supported on
$M$, we see that
$$ \langle \infdex(s, \alpha,\sigma), f\rangle =\int_{M}\int_{\mathfrak g}
e^{i s\Omega(x)} D\beta(x) \hat f(x)dx$$
$$=\int_{\mathfrak g}\int_{M} D\left(  e^{i s\Omega(x)}\beta(x)\right) \hat f(x)dx=0.$$

\bigskip

Finally, let  us consider two action forms $\sigma_1,\sigma_0$, with $\sigma_0=\sigma$.
Then the moment map for $\sigma_t=t\sigma_1+(1-t)\sigma_0$  is $\mu_t=t\mu_1+(1-t)\mu_0$, with $\mu_0=\mu$. We  assume that the closed set $\mu_t^{-1}(0)$ remains equal to $M^0$, for $t\in [0,1]$.
Let us see that
$\infdex(s,
\alpha,\sigma_1)=\infdex(s,\alpha,\sigma_0)$, for $s$ large.

Indeed, consider $\Omega(t)=D\sigma_t.$
Let $$I(t,s)=\int_M\int_{\mathfrak g} e^{i s\Omega(t,x)}
\alpha(x)\hat f(x)   dx.$$
 We obtain
$$-i\frac{d}{dt}I(t,s)=
s\int_M \int_{\mathfrak g} D(\sigma_1-\sigma_0)(x)e^{i
s\Omega(t,x)} \alpha(x)\hat f(x)   dx$$
$$= \nu+r$$
with $$\nu=s\int_{\mathfrak g} \int_M D\left( (\sigma_1-\sigma_0)
e^{i s\Omega(t,x )} \alpha(x )\right) \hat f(x) dx$$
 and
$$r= s\int_M  \int_{\mathfrak g}  (\sigma_1-\sigma_0)
e^{i s\Omega(t,x )} D( \alpha)(x )\hat f(x) dx .$$

As $\alpha(x)$ is compactly supported, $\nu=0$.

As for $r$, we remark that $\Omega(t,x)=\langle \mu_t,x\rangle +
q(t)$ where $q(t)$ is a two form. The integral $r$ involves the
value of $f$, and its derivatives,  at the points $s\mu_t(m)$.
As the compact support $K$ of $D\alpha$ is disjoint from $M^0$, our assumption implies that $\mu_t(m)$ is never equal to $0$ for $m\in K$ and $t\in [0,1]$. Thus, if $\rho:=\min_{m\in
K,t\in [0,1]}\|\mu_t(m)\|>0$, arguing as for Formulas \eqref{treo} and \eqref{tre}  we deduce that $r=0$ if we take  $s_0>R/\rho$.

\bigskip

One has still to verify that this linear map satisfies the
continuity properties that make it a  distribution. We leave this
to the reader.

In conclusion we have shown

\begin{theorem}\label{dein}
Let $\sigma$ be an action form with moment map $\mu$. Let $M^0=\mu^{-1}(0)$.
 Then we can
define a map
$$ \infdex_G^{\sigma}:H^*_{G,c}(M^0)\to \mathcal D'(\mathfrak g^*)^G$$
setting for any $[\alpha]\in H^*_{G,c}(M^0)$ and for any   smooth
function with compact support $f$ on $\mathfrak g^*$
$$\langle \infdex_G^{\sigma}([\alpha]),f\rangle:= \lim_{s\to \infty}
\int_{M}\int_{\mathfrak g}  e^{i s\Omega(x )}  \alpha(x ) \hat f(x
)dx .$$

The map $\infdex_G^{\sigma}$ is a well defined homomorphism of
$S[\mathfrak g^*]^G$ modules.

If the one form $\sigma$ moves along a smooth curve   $\sigma_t$ with moment map $\mu_t$  such that $\mu_t^{-1}(0)$ remains equal to $M^0$ , then
$$\infdex_G^{\sigma_t}=\infdex_G^{\sigma}.$$

\end{theorem}

In particular, if two action forms $\sigma_1,\sigma_2$ have same moment map $\mu$, the two infinitesimal indices $\infdex_G^{\sigma_1}$ and $\infdex_G^{\sigma_2}$ coincide.
Indeed, the moment map $\mu_t$  associated to $(1-t)\sigma_1+t\sigma_2$ is constant.
In view of this property, we denote simply by
$\infdex_G^{\mu}$ the map $\infdex_G^{\sigma}$. We call it the infinitesimal index map associated to $\mu$, or, for short,  the infdex map.

\begin{remark}
In general, the maps  $\infdex_G^{\mu}$ and $\infdex_G^{-\mu}$ are different (cf. Example \ref{variesemp}), although the zeroes of the moment maps associated to $\sigma$ and $-\sigma$ are the same.
Thus the stability condition that the set $\mu_t^{-1}(0)$ remains constant, when moving along $\sigma_t$,  is essential in order to insure the independence of the infinitesimal index.
\end{remark}

\bigskip

Let us give another formula for  $\infdex_G^{\mu}$. From this
formula, it will be clear that $\infdex_G^{\mu}$ belongs to the space $\mathcal S'(\g^*)^G$ of invariant tempered
distributions on $\mathfrak g^*$.

Let $f$ be a Schwartz function  on $\mathfrak g^*$. If $\alpha$ is
a representative of $[\alpha]\in H^*_{G,c}(M^0)$, we see that
$\int_{\mathfrak g} e^{i s \Omega(x)} (D\alpha)(x) \hat f(x) dx $
is a rapidly decreasing function of $s$: $D\alpha$ being
identically equal to $0$ on a neighborhood of $M^0$, this  is
expressed in terms of the value of the function $f$, and its
derivatives, at  points  $s\mu(m)$, where $\mu(m)$ is nonzero.
Thus we can define the compactly supported differential form
$\Phi(\alpha,f)$ on $M$ by

\begin{equation}\label{eq:ints}
\Phi(\alpha,f):= \int_{ \mathfrak g}\alpha(x) \hat f(x) dx+ i
\sigma \int_{s=0}^{\infty}\left(\int_{ \mathfrak g} e^{is
\Omega(x)} D\alpha(x) \hat f(x) dx\right)ds.
\end{equation}

\begin{proposition}\label{extend}
 We have
$$\langle \infdex_G^{\mu}(\alpha),f\rangle=  \int_{M} \Phi(\alpha,f) .$$
\end{proposition}
\begin{proof}
Let $f$ be a function with compact support on $\mathfrak g^*$.
Then $$\lim_{s\to \infty} \int_{M}\int_{\mathfrak g}  e^{i
s\Omega(x )} \alpha(x ) \hat f(x )dx$$ is equal to
 $$\int_{M}\int_{\mathfrak g}
\alpha(x ) \hat f(x )   dx+\int_0^{\infty} \frac{d}{ds}
\left(\int_{M}\int_{\mathfrak g} e^{is \Omega(x)} \alpha(x) \hat
f(x) dx\right) ds.$$ By Lemma \ref{lem:dds}, we  obtain the
proposition.
\end{proof}

\begin{remark}\label{par}
It is possible to define equivariant forms on $M$ with
$C^{-\infty}$ coefficients \cite{Kumar-Vergne}. Such a form is an
equivariant map from test densities on $\mathfrak g$ to
differential forms on $M$. The equivariant differential $D$
extends and we obtain the group $\mathcal H^{-\infty}_G(M)$, and
similarly the group  $\mathcal H^{-\infty}_{G,c}(M).$
 If $\alpha\in H_{G,c}^*(M^0)$, and
$g$ is a test function on $\mathfrak g$, we may define the
differential form
$$(p(\alpha),g dx)=\int_{\mathfrak g}\alpha(x)
g(x)dx+i\sigma \int_{s=0}^{\infty}(\int_{\mathfrak g} e^{i s
\Omega(x)} D\alpha(x)  g(x) dx) ds.
$$

It is easy to see that $p(\alpha)$ is a  compactly supported
equivariant form on $M$ with $C^{-\infty}$ coefficients  such that
$D(p(\alpha))=0$. Indeed,  we have $p(\alpha)=\alpha-\sigma
\frac{D\alpha}{D\sigma}$, where  $\frac{D\alpha}{D\sigma}$ is well
defined  in the distribution sense by $-i\int_{s=0}^{\infty}e^{i s
D\sigma} D\alpha \, ds$.

 We see that $\alpha\mapsto p(\alpha)$ defines a
map from $H_{G,c}^*(M^0)$ to $\mathcal H^{-\infty}_{G,c}(M).$ In
this framework, our distribution $ \infdex_G^{\mu}(\alpha)$ on $\mathfrak g^*$ is the Fourier
transform of the generalized function $\int_M p(\alpha)$ on
$\mathfrak g$.

Associated to  an action form $\sigma$,  Paradan defined a
particular element $P_\sigma \in\mathcal H^{-\infty}_G(M)$
representing $1$ and supported in a neighborhood of $M^0$
\cite{pep1}. This element is the form $p(1)$  defined above (when
$M^0$ is compact). Most of our subsequent theorems could be
obtained by Fourier transforms of Theorems proven in \cite{pep2},
\cite{par-ver2}  where basic functorial
properties of $P_\sigma$ are proved. However, we will work on
$\mathfrak g^*$ instead that on $\mathfrak g$ and  we will give
direct proofs.

\end{remark}

\subsection{Extension of the definition of the infinitesimal index}
 Let us see that  the definition of  the infinitesimal index extends  to $\mathcal H^{\infty,m}_{G,c}(M^0)$.

If $\alpha\in {\mathcal A}^{\infty,m}_{G,c}(M )$  is such that
$D\alpha=0$ in a neighborhood of $M^0$, we see that Lemma
\ref{lem:dds} still holds, $f$ being  a Schwartz function on
$\mathfrak g^*$:
$$-i\frac{d}{ds} \int_{M}\int_{\mathfrak g} e^{is
\Omega(x)} \alpha(x) \hat f(x) dx =\int_{M}\int_{\mathfrak g}
e^{is \Omega(x)} \sigma D\alpha(x) \hat f(x) dx.$$
 Since $\alpha$
is of at most polynomial growth, the function of $x$ given by
$D\alpha(x) \hat f(x)$ is still a Schwartz function of $x$. Thus
by Fourier inversion, we again see that $-i\frac{d}{ds}
\int_{M}\int_{\mathfrak g} e^{is \Omega(x)} \alpha(x) \hat f(x)
dx$ is a rapidly  decreasing function of $s$ and we may define

$$\langle \infdex_G^{\mu}(\alpha),f\rangle =  \lim_{s\to\infty}
\int_{M}\int_{\mathfrak g} e^{i s\Omega(x )}  \alpha(x ) \hat f(x)
dx.$$

We have again the formula:

$$\langle \infdex_G^{\mu}(\alpha),f\rangle = \int_M \Phi(\alpha,f)$$
where $\Phi(\alpha,f)$ is given by Equation (\ref{eq:ints}).

 This formula shows that
$\infdex_G^{\mu}(\alpha)$ is a $G$-invariant tempered distribution
on $\mathfrak g^*$. With similar arguments, we obtain the following theorem.

\begin{theorem}\label{theo:infdexgrowth}
We can define a map
$$ \infdex_G^{\mu}:\mathcal H^{\infty,m}_{G,c}(M^0)\to \mathcal S'(\mathfrak g^*)^G$$
setting for any $[\alpha]\in \mathcal H^{\infty,m}_{G,c}(M^0)$ and
for any  Schwartz function  $f$ on $\mathfrak g^*$
$$\langle \infdex_G^{\mu}([\alpha]),f\rangle:= \lim_{s\to \infty}
\int_{M}\int_{\mathfrak g}  e^{i s\Omega(x )}  \alpha(x ) \hat f(x
)   dx .$$

If $\sigma$ moves smoothly along a curve $\sigma_t$ such that $\mu_t^{-1}(0)$ remains equal to $M^0$, the map
 $\infdex_G^{\mu_t}$ remains constant.

Furthermore, using the Fourier transform $\mathcal F$ of tempered
distributions
\begin{equation}\label{icl}
\mathcal F( {\rm infdex}_G^{\mu}(
[\alpha]))=\lim_{s\to\infty} \int_{M} e^{i s\Omega(x )}  \alpha(x
).
\end{equation}

\end{theorem}

\begin{remark}\label{rem:stableinfdex}
If $f$ is with compact support and the Fourier transform of
$\alpha(x)$ is a distribution with compact support on $\mathfrak g^*$, the
value $ \int_{M}\int_{\mathfrak g}  e^{i s\Omega(x )}  \alpha(x )
\hat f(x )   dx $ is independent of $s$ when $s$ is sufficiently
large.

\end{remark}

Let us state some immediate properties of the infinitesimal index.
We recall that our construction of the infinitesimal index map is strongly inspired by      Witten  nonabelian localization
theorem  \cite{wit}. In particular, we have the following ``nonabelian localization'' result.

\begin{theorem} \label{theowit1}
 Let $[\alpha]\in {\mathcal H}^{\infty,m}_{G,c}(M) $
and   $I(x)=\int_M \alpha(x)$, a function on $\mathfrak g$ with moderate growth.
Let $\sigma$ be an action form, and let $M^0$  be the zeroes of the
moment map. Then $[\alpha]$ defines an element $[\alpha_0]$  in
${\mathcal H}_{G,c}^{\infty,m}(M ^0) $ and
\begin{equation}
\label{Wl} \mathcal F(
\infdex_G^{\mu}([\alpha_0]))(x)=I(x).
\end{equation}
 \end{theorem}

\begin{proof}
This is clear from Formula \eqref{icl} as $\int_M e^{is\Omega(x)} \alpha(x)$ does not depend
on $s$, as $\Omega(x)$ is exact and $\alpha$ is closed with compact support.
\end{proof}

The left hand side of
\eqref{Wl} depends only of the restriction of $\alpha$ on a small
neighborhood of $M^0$. Thus Theorem \ref{theowit1} says that we can compute the equivariant integral of $\alpha$ on $M$, knowing $\alpha$ on a small neighborhood of $M^0$.

\begin {remark}  Let $M$ be a $G$-manifold equipped with a $G$ invariant Riemannian metric. Take a  $G$-invariant vector field $V$ on $M$ so that $V_m$ at each point $m\in M$ is tangent to the orbit $Gm$ and let  $\sigma$ be the one form associated to $V$ using the metric.
 Then $M^0$ is the set of zeroes  of the vector  field $V$.

$\bullet$ If $G$ is abelian,  we may
choose $V=v_x$ with $x$ generic in $\mathfrak g$, and  then
$M^0=M^G$, the set of fixed points of $G$ on $M$. Theorem \ref{theowit1} leads to
the "abelian localization theorem" of Atiyah-Bott-Berline-Vergne
\cite{ati-bot84},\cite{ber-ver82}.

 $\bullet$ When $G$ is non necessarily  abelian and $M$ is provided with an Hamiltonian structure with symplectic moment map $\nu: M\to \mathfrak g^*$,  then the Kirwan vector field $V_m=\exp (\epsilon \nu(m)) m$ is such that $M^0$ coincides with the critical points of the function $\|\nu\|^2$ (we used  an identification $\mathfrak g^*=\mathfrak g$).
Then one of the connected components of $M^0$ is the zeroes of the symplectic moment map $\nu$, and $\mu$ and $\nu$ coincide near this component.
This is  the situation considered by Witten (and extensively studied by Paradan, \cite{pep1})  with  applications to intersection numbers of  reduced spaces $\nu^{-1}(0)/G$ (as in \cite{jef-kir2}). \end{remark}

\begin{example}\label{variesemp}

$\bullet$\label{Gtrivial} If $G:=\{1\}$ is trivial,
$H_{G,c}^*(M^0)$ is equal to $H_c^*(M)$ and the infinitesimal index  maps to constants, by   just
integration of compactly supported cohomology classes.

 $\bullet$\label{point}
 If $M=\{pt\}$ is a point,
 the moment map and $\Omega(x )$  are both 0 while $M^0=M=\{pt\}$.
 Its equivariant cohomology is  $S[\mathfrak g^*]^G$.

 By Proposition \ref{dein} it is then enough to compute the infinitesimal index
 of the class 1.
  This is given by $$f\mapsto \int_{\mathfrak g} \hat f(x )    dx=f(0)$$
  by Fourier inversion formula.
  So,  in this case the infinitesimal index of 1 is the $\delta$--function $\delta_0$.

More generally,  we  have  extended the definition of
$\infdex_G^{\mu}$ to the space
 $\mathcal P^\infty(\mathfrak g)^G$ of invariant functions on $\mathfrak g$ with at most polynomial growth.
  If $\alpha(x)$ is any $G$-invariant function on $\g$ with polynomial growth
  and $\hat \alpha$ its Fourier transform (a distribution on $\mathfrak g^*$),
   we obtain

 \begin{equation}\label{pointfourier}
 \infdex_G^{\mu^0}(\alpha)=\hat \alpha.
 \end{equation}

$\bullet$\label{esse1} Consider $M=T^*S^1$ with the canonical
action form as in Example \ref{ts}. Then $M^0=S^1$. We compute the
infinitesimal index of the class $1\in H^*_{G,c}(M^0)=\mathbb R$.
Let $\chi(t)$  be a function identically equal to $1$ in a
neighborhood of $t=0$. Then $D\sigma(x)=x t+ dt\wedge  d\theta$,
and by definition
$$\langle \infdex_G^{\mu}(1),f\rangle = \lim_{s\to \infty} \int_{T^*S^1}
\Big (\int_{-\infty}^\infty\chi(t) e^{isx t+ is dt d\theta}\hat
f(x) dx\Big )$$
$$= \lim_{s\to \infty} \int_{T^*S^1}\chi(t) f(st)  e^{is dt d\theta} =\lim_{s\to \infty}2\pi is\int_{\mathbb R}\chi(t) f(st)dt $$
$$=\lim_{s\to \infty}2\pi i \int_{\R}\chi(t/s) f(t)  dt.$$

Passing to the limit, we see that
$$\langle \infdex_G^{\mu}(1),f\rangle=2\pi i \int_{\R} f(t)  dt,$$
that is  the distribution $\infdex_G^{\mu}(1)$ is just $2\pi i$ times  the integration
against the Lebesgue measure $dt$.

$\bullet$\label{esse2} More generally, consider $M=T^*G$ with the canonical
action form $\sigma$ as in Example \ref{ts} and the canonical $G\times G$ action by left and right multiplications.
Set $r:=\dim G$ and  orient  $M$ via the Liouville form $d\sigma^r$. We take $(g,\zeta)$ with $g\in G$ and $\zeta\in \mathfrak g^*$ as coordinates on $M=G\times \mathfrak g^*$.
We write an element of $\mathfrak g\oplus \mathfrak g$ as $(y,x)$.
We have  $M^0=G$ and want to compute the
infinitesimal index of the class $1\in H^0_{G\times G,c}(M^0)=\mathbb R$.
  Let $\chi $  be a function with compact support, $G $ invariant and identically equal to $1$ in a
neighborhood of $0$ in $\mathfrak g^* $.
This function gives  also a function on $T^*G=G\times \mathfrak g^*$,
 still denoted by $\chi$. Then $\chi(g,\zeta)=\chi(\zeta)$ is a representative of 1.
Let $f$ be a function on $\mathfrak g^*\oplus \mathfrak g^*$.
Then  (using Formulae (\ref{mmm}), (\ref{liofo1}) and Fourier inversion), we have: $$\langle \infdex_{G\times G}^{\mu}(1),f\rangle = \lim_{s\to \infty} \int_{T^*G}\int_{\mathfrak g\oplus \mathfrak g}
\chi(\zeta) e^{is d\sigma}e ^{i s\langle \zeta, g^{-1}y-x\rangle} \hat f(x,y) dx dy$$
$$= \lim_{s\to \infty} \int_{T^*G}\int_{\mathfrak g\oplus \mathfrak g}
\chi(\zeta) (is)^{r} \frac{(d\sigma)^r}{r!} e ^{i s   \langle \zeta, g^{-1}y-x\rangle} \hat f(x,y) dx dy$$
$$=i^r
\lim_{s\to \infty} \int_{G\times \mathfrak g^*}\int_{\mathfrak g\oplus \mathfrak g}
\chi(\zeta) s^{r} e ^{i s\langle \zeta, g^{-1}y-x\rangle} \hat f(x,y) dx dy d\zeta dg$$
$$=i ^r \lim_{s\to \infty} \int_{ G\times \mathfrak g^*}
 \chi(\zeta) f ( sg\zeta, - s\zeta) d\zeta dg$$
 $$= i ^r \lim_{s\to \infty}\int_{ G\times \mathfrak g^*}
 \chi(\zeta/s) f (g\zeta, - \zeta) d\zeta dg.$$

 Taking limit, we obtain
 \begin{equation}\label{infdex1}
  \langle \infdex_{G\times G}^{\mu}(1),f\rangle  =  i^r  \int_{G\times \mathfrak g^*}
 f ( g\zeta,  -\zeta)   d\zeta dg
\end{equation}

$\bullet$\label{Heaviside}
 Consider now $M=\mathbb R^2$ as in
Example \ref{v}. As we have seen, the action form $\sigma$ is
$\frac{1}{2} (v_1 dv_2-v_2 dv_1)$, so $D\sigma(x) =dv_1\wedge
dv_2+x\|v\|^2/2$. Then $M^0=0$.

We compute the infinitesimal index of the class $1\in
H^*_{G,c}(M^0)$. Let $\chi(t)$  be a function on $\R$ with compact
support and  identically equal to $1$ in a neighborhood of $t=0$.
We then get
$$\langle \infdex_G^{\mu}(1),f\rangle =
\lim_{s\to \infty} \int_{\R^2}
\int_{-\infty}^\infty\chi(\|v\|^2) e^{i s x
\frac{\|v\|^2}{2}+ i s dv_1 dv_2}\hat f(x) dx.$$

 Using
polar coordinates on $\R^2$, and  inversion of Fourier transform, we see that
$$\langle \infdex_G^{\mu}(1),f\rangle=2\pi i \lim_{s\to \infty}\int_0^{\infty}\chi(t/s)\int_0^\infty f(t)dt.$$

Taking the limit, we obtain
$$\langle \infdex_G^{\mu}(1),f\rangle=2\pi i\int_0^\infty f(t)dt,$$
that is  the distribution $\infdex_G^{\mu}(1)$ is just  $2\pi i$ times  the Heaviside
distribution supported on $\R^+.$

\end{example}

\section{Properties of the infinitesimal index}

There are several functorial properties of the infinitesimal index
that we need to investigate: {\em  locality, product, restriction, the map $i_!$, free action}.

\subsection{Locality}\label{excision}

The easiest property is  {\em locality}.

Let $M$ be a $G$--action manifold  with moment map $\mu$ and
$i:U\to M$ an invariant open set, then we have a mapping
$i_*:\mathcal A_{G,c}(U)\to \mathcal A_{G,c}(M)$ which induces
also a mapping
$$i_*:H^*_{G,c}(U^0)\to H^*_{G,c}(M^0).$$
\begin{proposition}\label{loc}
The mapping $i_*$ is compatible with the infinitesimal index.
\end{proposition}
\begin{proof}
This is immediate from the definitions.
\end{proof}

\subsection{Product of manifolds}

  If we have a product  $M_1\times M_2$ of two manifolds
 relative to two different groups $G_1\times G_2$,
 we have
$$(M_1\times M_2)^0= M_1^0\times M_2^0$$
and the cohomology is the product.
\begin{proposition}\label{product}
The infinitesimal index of the external product of two cohomology
classes is the external product of the two distributions.
\end{proposition}
\begin{proof}
This is immediate from the definitions.
\end{proof}

\subsection{ Restriction to subgroups}\label{subgroup}

Let $L\subset G$ be a compact subgroup of $G$ so that $\mathfrak
l$, the Lie algebra of $L$, is  a subalgebra of $\mathfrak g$. The
moment map $\mu_L$ for $L$ is just the composition of $\mu_G$ with
the restriction $p:\mathfrak g^*\to \mathfrak l^*$. Thus
$\mu_L^{-1}(0)\supset \mu_G^{-1}(0).$

 If $f$ is a test function on $\mathfrak l^*$, then $p^*f$ is a
smooth function on $\mathfrak g^*$ constant along the fibers of the
projection. \begin{definition}
We will say that a distribution $\Theta$ on $\mathfrak
g^*$
 is a distribution with
compact support along the fibers, if for any test function $f$ on
$\mathfrak l^*$, the distribution $(p^*f)\Theta$ is with compact
support on $\mathfrak g^*$.
\end{definition} If $\Theta$  is a distribution on $\mathfrak
g^*$
 with
compact support along the fibers,  we may define $p_*\Theta$ as a
distribution on $\mathfrak l^*$ by
\begin{equation}
\label{plos}\langle p_*\Theta,f\rangle :=\int_{\mathfrak g^*} (p^*f) \Theta.
\end{equation}

The right hand side is computed as the limit when $T$ tends to
$\infty$ of $\langle \Theta,(p^*f)\chi_T \rangle $ when $\chi_T$
is a smooth function with compact support and equal to $1$ on the
ball $B_T$ of $\mathfrak g^*$.

Let $Z_G$ be a closed $G$-invariant subset of $M$ containing
$\mu_L^{-1}(0)$ (if $L$ is normal in $G$ in particular if $G$ is abelian, we can take $Z_G=\mu_L^{-1}(0)$). Then we have two maps

$$j:
H^*_{G,c}(Z_G)\to H^*_{G,c}(\mu_G^{-1}(0))$$ and
$$r:H^*_{G,c}(Z_G)\to H^*_{L,c}(\mu_L^{-1}(0)).$$
\begin{theorem}\label{restriction}
If $[\alpha]\in H^*_{G,c}(Z_G)$
 then ${\rm infdex}_G^{\mu_G}(j[\alpha])$ is compactly supported along the
 fibers of the map $p:\mathfrak g^*\to \mathfrak l^*$, and
\begin{equation}
\label{cafor}p_*({\rm infdex}_G^{\mu_G}(j[\alpha]))={\rm infdex}_L^{\mu_L}(r[\alpha]).
\end{equation}
\end{theorem}

\begin{proof}
Write  $\mathcal F^{\mathfrak g^*}(h)$ for the Fourier
transform $\hat h$ of a function $h$ on $\mathfrak g^*$.

Let $f$ be a test function  on $\mathfrak l^*$ with support on a
ball $B_R$. We have, for $\chi$ a test function
on $\mathfrak g^*$,
$$\langle (p^*f) {\rm
infdex}_G^{\mu_G}(j[\alpha]), \chi\rangle = \lim_{s\to
\infty}\int_M\int_{\mathfrak g} e^{i s\Omega(x)}\alpha(x)
\mathcal F^{\mathfrak g^*}((p^*f) \chi)(x) dx.$$

By our assumption on $\alpha$, there exists $\epsilon >0$ such
that $D\alpha$ is equal to $0$ on the subset
$\|\mu_L(m)\|=\|p\mu_G(m)\|<\epsilon $ of $M$. The  support  $C$ of  $(p^*f) \chi$  is contained in the set  of $y\in\mathfrak g^*$  such that $\|p(y)\|<R$.  The support $K$ of $D\alpha$  is contained in the set of points $m$ such that $\|p\mu_G(m)\|>\epsilon $. Thus by Proposition \ref{ilsupd}  and the argument of Lemma  \ref{lem:dds}, the
distribution  $$\chi\to \int_M \int_{\mathfrak g}e^{i
s\Omega(x)}\alpha(x) \mathcal F^{\mathfrak g^*}(\chi p^*f)(x) dx$$
stabilizes as soon as $s>R/\epsilon$.

Write for $s_0>R/\epsilon$ $$\langle (p^*f){\rm infdex}_G^{\mu_G}(j[\alpha]),
\chi\rangle =\int_M \int_{\mathfrak g}e^{is_0 \Omega(x)}\alpha(x)
\mathcal F^{\mathfrak g^*}(\chi p^*f)(x) dx$$
$$=\int_M \Psi(s_0,\alpha,\chi p^*f)$$ where $$\Psi(s_0,\alpha,\chi p^*f)(m)=
\sum_a\Big[ \int_{\mathfrak g}  P_a(x)e^{i s_0 \langle
\mu(m),x\rangle } \mathcal F^{\mathfrak g^*} ( \chi p^*f)(x) dx \Big]
e^{i s_0d\sigma}\alpha^a  $$ \begin{equation}
\label{checca}= \sum_a((P_a(-i\partial)( \chi p^*f )\circ (s_0\mu) )
e^{i s_0d\sigma}\alpha^a.
\end{equation} Applying Proposition \ref{ilsupd},  we have
that, if $K$ is the compact
support of $\alpha$, as $s_0$ is greater than $R/\epsilon$,  the form $ \Psi(s_0,\alpha,\chi p^*f)$ is supported on the compact
subset $s_0 \mu_G(K)$ in $\mathfrak g^*$. This shows the first statement that  ${\rm infdex}_G^{\mu_G}(j[\alpha])$ is compactly supported along the  fibers of $p$.

We pass next to Formula \eqref{cafor}.   We then have $$\langle (p^*f) {\rm
infdex}_G^{\mu_G}(j[\alpha]), \chi_T\rangle=\int_M \Psi(s_0,\alpha,\chi_T
p^*f)$$ for any $T$ large.

 Using Formula (\ref{checca}), when $T$ is sufficiently large, as
 $\chi_T$ is equal to $1$ on   the compact
subset $s_0 \mu_G(K)$, thus
 $\Psi(s_0,\alpha,\chi_T p^*f)$ is simply
$$\sum_a((P_a(-i\partial)  p^*f)\circ (s_0\mu) ) e^{i
s_0 d\sigma}\alpha^a.$$

As $p^*f$ is constant along the fibers, if we denote by $\alpha_0$
the restriction of $\alpha(x)$ to $\mathfrak l$, we see that
$\Psi(s_0,\alpha,\chi_T p^*f)$ is equal to the differential form
 $\Psi(s_0,\alpha_0,f)$ as all derivatives in the $\ker p$ direction
 annihilate $p^*f$.
We thus obtain our theorem.

\end{proof}

\subsection{Thom class and the map $i_!$ }\label{push}

 Let $Z$ be an  oriented $G$ manifold of
dimension $d$ and $i:M\hookrightarrow Z$ a $G$-stable oriented submanifold of
dimension $n=d-k$.
 Assume that $M$
 is an action  manifold with moment map  $\mu$ and that  $Z$ is equipped with an action form $\sigma_Z$ such that the associated moment map $\mu_Z$ extends $\mu$. Thus $Z^0\cap M=M^0$.
 Under these assumptions,
    we will define a map
$$i_!:H^*_{G,c}(M^0)\to H^*_{G,c}(Z^0)$$
preserving the infdex.

Let us recall the existence of an  equivariant Thom class (\cite{Mathai-Quillen}, see
\cite{gui-ste99} pag. 158, \cite{par-ver1}).
We assume first that $M$ has a $G$-stable tubular
neighborhood $N$ in $Z$, with projection $p:N\to M$.  Then there exists a unique class  $\tau_M$ of
equivariantly closed forms   on $N$ with compact support along the
fibers  so that the integral $p_*\tau_M$  is identically equal to
$1$ along each fiber. Thus for any equivariant form $\alpha(x)$ on
$M$ with compact support (but not necessarily closed),
 we have that $$\int_M\alpha =\int_{   N}  p^*\alpha\wedge \tau_M.$$

 In general, let us take a class $[\alpha]\in H^*_{G,c}(M^0)$ where
$\alpha\in  \mathcal A_{G,c}(M)  $ and $D\alpha$ has support $K$
in $M\setminus M^0$.

Consider  a $G$-stable open set $U\subset M$ with the following
properties.
\begin{enumerate}
\item The support of $\alpha$ is contained in $U$. \item  The
closure of $ U$ is compact and has an open  neighborhood  $A$ in
$Z$ such that $M\cap A$ has a G-stable tubular neighborhood in
$A$.
\end{enumerate}

By locality,  we can then substitute $U$ to $M$ and thus assume
that the pair $(Z,M)$ has all the properties which insure the
existence of a Thom class $\tau_M$.

Consider a $G$-invariant Riemannian metric on the normal bundle
$\mathcal N$ to $M$ in $Z$.
 Define $S^{\epsilon}$ as the (open) disk bundle of radius $\epsilon$
 in $\mathcal N$.
  Then we can  take our tubular neighborhood
  in such a way that it is diffeomorphic to $S^{\epsilon}$
  for some $\epsilon$.

   We claim that we can  take
 $S^\epsilon$ so close to $ M$  that
 $p^{-1}K\cap S^\epsilon \cap Z^0=\emptyset$.
 Indeed,  $p^{-1}K\cap \overline{S^{\epsilon}}$
 is a compact set and, since $K$ is disjoint from $M^0$ and
 hence from $Z^0$, for a sufficiently small $\epsilon$,
    $p^{-1}K\cap \overline{S^\epsilon}$ is disjoint from $Z^0$.
   \smallskip

    Let us now fix  the Thom form  $\tau_M$ in
    $\mathcal A_{G,c}(\mathcal N)$
    with support  in  $S^\epsilon$.

 Consider then the form $p^*\alpha\wedge \tau_M$.  We have
 that
 $ D(p^*\alpha\wedge \tau_M)=p^*D\alpha\wedge \tau_M$
 has support in $p^{-1}K\cap S^\epsilon \subset Z\setminus Z^0$. It follows that $p^*\alpha\wedge \tau_M$ defines an element in
$H^*_{G,c}(Z^0)$.

We claim that this element depends only on the class $[\alpha]$.
So first take another Thom form $\tau'_M$ with the same
properties. Then there is a form $r_M\in  \mathcal
A_{G,c}(S^\epsilon)  $   so that
$\tau_M-\tau'_M=Dr_M$ and
$$p^*\alpha\wedge \tau_M-p^*\alpha\wedge \tau'_M=
p^*\alpha\wedge Dr_M= D(p^*\alpha\wedge  r_M)-p^*D\alpha\wedge r_M
$$ where $p^*\alpha\wedge  r_M $ has compact support   and
$p^*D\alpha\wedge  r_M$
 has support in  $Z\setminus Z^0$.

Next  assume that $\alpha$ is supported outside  $M^0$,
 then again
we may take $\tau_M$ so that $p^*\alpha\wedge \tau_M$  is
supported outside $Z^0$.

Finally, if $\alpha=D\beta$, we have
$p^*\alpha\wedge\tau_M=D(p^*\beta \wedge \tau_M)$.

Hence we can set
\begin{equation}
\label{ilosc}i_![\alpha]:=[p^* \alpha\wedge \tau_M].
\end{equation}
\begin{theorem}\label{ilsh}

Assume that $M$
 is an action  manifold with  action form $\sigma$ and moment map  $\mu$ and that  $Z$ is equipped with an action form $\sigma_Z$ such that the associated moment map $\mu_Z$ extends $\mu$.
 Then the morphism
$$i_!:H^*_{G,c}(M^0)\to H^*_{G,c}(Z^0)$$
preserves the infinitesimal index.

\end{theorem}

 \begin{remark}
 We do not need to assume that the restriction of $\sigma_Z$ to $M$ is the action form $\sigma$ on $M$, only that the moment map $\mu_Z$ restricts to $\mu$.
 \end{remark}

 \begin{proof}
 First let us  see that   ${\rm infdex}_G^{\mu_Z}(i_![\alpha])$ does
not depend of the choice of the form $\sigma_Z$ on $Z$, if the moment map $\mu_Z$ restricts to $\mu$.
We can assume $Z=N$. Let $\beta= p^*\alpha\wedge
\tau_M$. The form $\beta$ is compactly supported.

Let $\sigma_1,\sigma_0$ be two one forms on $Z$ and  consider
$\sigma_t=t\sigma_1+(1-t)\sigma_0$ and $\mu_t$ the corresponding
moment map.  Set $\Omega(t)=D\sigma_t$. We assume that the map $\mu_t$
coincides with $\mu$ on $M$ for all $t$.
 Thus, provided we choose $\tau_M$ with support sufficiently close to  $M$,  there exists an $h>0$  such  that on the support of $D\beta$, we have $\|\mu_t\|>h>0$.

Define $$I(t,s):= \int_N\int_{\mathfrak g} \beta (x)e^{i
s\Omega(t,x)}  \beta (x) \hat f(x) dx.$$

     We can prove that $\frac{d}{dt}I(t,s)=0$ in  the same way that the  invariance of the infinitesimal index $\infdex_G^{\mu_t}$  along a smooth curve $\mu_t$  (proof  of Theorem \ref{dein}), thus we skip the proof.

\smallskip

Having established the independence from $\sigma$, we choose for the final computation    $\sigma_Z:=p^*\sigma$. In this case, since
$\beta=p^*\alpha\wedge \tau_M$,
\begin{equation}
\label{ilsf}\langle {\rm infdex}_G^{\mu_Z}([\beta]), f\rangle=
\lim_{s\to \infty}\int_{\mathfrak g} \int_{N} p^*\left(e^{is
\Omega(x)} \alpha(x) \right)\wedge \tau_M(x) \hat f(x)dx .
\end{equation} As
$\tau_M$ has integral $1$ over each fiber of the projection $p:
N\to M$, we obtain that \eqref{ilsf} is equal to
$$\lim_{s\to \infty}\int_{\mathfrak g}
\int_{M}  e^{is \Omega(x)}  \alpha(x)\hat f(x)dx  $$ which is
our statement.
\end{proof}

\subsection{Free action}\label{sub:fre}

Let $G$ and $L$ be two compact groups. Consider now an oriented manifold
 $N$ under  $G\times L$ action, with action one form $\sigma_{N}$ and
 moment map $\mu_{G\times L}=(\mu_G,\mu_L): N\to \mathfrak g^*\oplus \mathfrak l^*  $.  We
set $N^0=\mu_{G\times L}^{-1}(0)$.

 Assume that

  $\bullet$  the group  $L$ acts freely on $N$.

 $\bullet$ $0$ is a regular value of $\mu_L$.

\medskip

Define  $P=\mu_L^{-1}(0)$. By assumption $P$ is a manifold with a
free $L$-action so  $$M:=\mu_L^{-1}(0)/L$$ is a $G$-manifold.
 We will see in a short while that the orientation on $N$ determines a natural orientation on $M$.

 We denote by $\pi$ the projection $\pi:P\to M$. The invariance of $\sigma_{N}$ under  $L$ action then implies
\begin{proposition}\label{prop:sigmabar}
The restriction   $\overline \sigma$ of $\sigma_{N}$ to $P$ verifies $\iota_x \overline \sigma=0$ for
any $x\in \mathfrak l$ and   descends to a
 $G$-invariant action form $\sigma_M$ on $M$, thus $M$ is an action manifold and $\bar\sigma=\pi^*(\sigma_M)$.
\end{proposition} We denote  by $\mu$ the moment map on $M$
 associated to $\sigma_M$. The map $\mu$ is
 obtained factoring   the restriction of $\mu_G$ to $P$ which is $L$ invariant, that is $\mu_G=\mu\circ \pi$ on $P$.
  Since $N^0$  is the subset of $P$ where $\mu_G$ equals 0, we see that $M^0=\mu^{-1}(0),$ the fiber at zero of $\mu,$ is $M^0=N^0/L$.

\medskip

 Recall (Proposition \ref{pro:free}) that since the action of $L$ is free,
 we have an isomorphism $\pi^*:H^*_{G,c}(M^0)
 \to H^*_{G\times L,c}(N^0)$.

 Our goal in this section is, given a class $[\gamma]\in H_{G,c}^*(M^0)$,   to compare $ \infdex_G^{\mu}([\gamma]) $ and $ \infdex_{G\times L}^{ \mu_{G\times L}}(\pi^*([\gamma]))$.
\smallskip

As $0$ is a regular value of $\mu_L$, any $L$-stable compact
subset $K$ in $P$ has an $L$-stable  neighborhood   in $N$
isomorphic to $K\times \mathfrak l^*$ with moment map $\mu_L$ being
the  projection on the second factor.
 Since the computations of the infinitesimal index of a given class
 with compact support are local around $N^0$  (by Proposition
 \ref{loc}),
  we may assume that $N=P\times \mathfrak l^*$ and that the  moment map $\mu_L$  is the  projection on the second factor.
  We write an element of $N$ as $(p,\zeta)$ with $p\in P$, $\zeta\in \mathfrak l^*$.

  The composition  of the projection $\eta:N=P\times \mathfrak l^*\to P$ and
  of $\pi:P\to M$   is a fibration  with fiber $L\times \mathfrak l^*=T^*L$.  We orient $M$ using the orientation of $N$ and that given by the symplectic structure on $T^*L$  (see  Formula \eqref{liofo}).

\subsubsection{An auxiliary form}
Let us choose now a connection form $\ \omega\in \mathcal A^1(P)\otimes \mathfrak l$ for the free action of $L$ on $P$. We want to apply Definition \ref{ary}  to the following functions.
 For  $\zeta$   a point in $\l^*$,
  define  $\theta_\zeta \in C^{\infty}(\mathfrak l)^L$  by
  $$\theta_\zeta(x):=\int_{L}e^{i\langle x,l\zeta\rangle
  }dl= \int_{\mathfrak l^*} e^{i\langle f,x\rangle} d\beta_\zeta(f)$$
  where $dl$ is a  Haar measure on $L$ or, in an equivalent way
   where $d\beta_\zeta(f)$ is  a $L$-invariant measure on the orbit $L\zeta\subset \mathfrak l^*$.
\smallskip

Thus for any $\zeta\in \l^*$, we may consider, using  the curvature $R$, cf. Formula \eqref{ledc},  the $G$-equivariant
closed form  $\theta_\zeta(R_y)$ on $M$ given by
\begin{equation}
\label{thzr}\theta_\zeta(R_y)=\int_{L}e^{i\langle R_y, l\zeta\rangle
  }dl=\int_{L}e^{i\langle -\iota_y \omega, l\zeta\rangle}
  e^{i\langle R,l\zeta\rangle
  }dl.
\end{equation}

We need some growth properties of the function $y\to \theta_\zeta(R_y)$.
If we fix $p\in P$ and $\zeta\in \mathfrak l^*$,   let us see that
\begin{lemma}
The function $y\to \theta_\zeta(R_y)(p)$ is the Fourier transform of a compactly supported measure $d\mu_{p,\zeta}$ on $\mathfrak g^*$ (with values in  $\bigwedge T_p^*P$).
\end{lemma}\begin{proof}
Indeed, let $f\in \mathfrak l^*$.
The function $\langle -\iota_y\omega(p),f\rangle$ is linear in
$y\in \mathfrak g$, so we  write  $\langle
-\iota_y\omega(p),f\rangle= \langle y, h(p,f)\rangle$
with $h(p,f)\in \mathfrak g^*$ depending smoothly on
$p,f$.
We see that $$\theta_\zeta(R_y)(p)=\int_{\mathfrak l^*} e^{i\langle y,h(p,f)\rangle}  e^{i\langle f, R\rangle } d\beta_\zeta(f),$$where $d\beta_\zeta(f)$ is  a $L$-invariant measure supported  on the orbit $L\zeta\subset \mathfrak l^*$.

Let us integrate over the fiber of the map $h_p:\mathfrak l^*\to \mathfrak g^*$ given by $f \to h(p,f)=\xi$. We obtain that

\begin{equation}\label{fou}
\theta_\zeta(R_y)(p)=\int_{\mathfrak \g^*} e^{i\langle y,\xi\rangle}  (h_p)_* (e^{i\langle f, R\rangle } d\beta_\zeta(f)).
\end{equation}

In this formula,  $(h_p)_{*} (e^{i\langle f, R\rangle } d\beta_\zeta(f))$
 is a measure supported on the compact set  $h_p(L\zeta)$ as $d\beta_\zeta(f)$ is supported in the compact set $L\zeta$.
In particular, we see that, over a compact subset of $P$,   $y\to \theta_\zeta(R_y)(p)$ is a bounded function of $y$ as well as all its derivatives in $y$ and   estimates are uniforms in $\zeta$ if $\zeta$ varies in a compact set of $\mathfrak l^*$.

\end{proof}

If $[\gamma] \in H_{G,c}^*(M^0)$, we choose a representative $\gamma(y)$ which is a form with compact support on $M$ and depending of $y$ in a polynomial way. Set \begin{equation}
\tilde
\gamma_\zeta(y):=\gamma(y)\theta_\zeta(R_y).
\end{equation}
\begin{proposition} The equivariant form $ \tilde
\gamma_\zeta(y)$
is   of  at most polynomial
growth in $y$. It represents a class in $\mathcal H_{G,c}^{\infty,m}(M^0)$ which does
not depend of the choice of the connection $\omega$ but only
on the choice of the Haar measure $dl$.
\end{proposition}

\begin{proof}

The fact that   $\tilde
\gamma_\zeta(y)$  is of at most polynomial growth follows from the preceding discussion. The second statement is proved as  in
(\cite{ber-ver83},\cite{bot-tu}, see \cite{ber-get-ver}).

 \end{proof}

  Remark that  $\theta_0(R_y)={\rm vol}(L,dl) $ where
${\rm vol}(L,dl)$ is the volume of the compact Lie group $L$ for
the Haar measure $dl$ such  that $dl d\zeta$ is the canonical measure
on $T^*L=L\times \mathfrak l^*$ (by right or left trivialization).

\subsubsection{The main formula}
With the notations of the previous paragraph,
given  $[\gamma] \in H_{G,c}^*(M^0)$, we may apply the infinitesimal index
construction  (Theorem \ref{theo:infdexgrowth}) to the cohomology
class $[\tilde \gamma_\zeta]\in \mathcal H_{G,c}^{\infty,m}(M^0)$ of
the equivariant form $\tilde
\gamma_\zeta(y)=\gamma(y)\theta_\zeta(R_y) $ we have:

\begin{theorem}\label{theo:free}
Let $f_1$ be a test function on $\mathfrak l^*$ and $f_2$ be a
test function on $\mathfrak g^*$. Then $\langle
\infdex^{\mu}_G([\tilde\gamma_{\zeta}]),f_2\rangle$ is a smooth
function of $\zeta$ and
\begin{equation}
\label{inni}\langle \infdex_{G\times L}^{ \mu_{G\times L}}(\pi^*([\gamma])),f_1f_2\rangle =i^{\dim L}
\int_{\l^*} \langle
\infdex^{\mu}_G([\tilde\gamma_{\zeta}]),f_2\rangle  f_1(\zeta)
d\zeta.
\end{equation}\end{theorem}

\bigskip

\begin{remark}
Formula (\ref{infdex1}) is a particular case of the above theorem. Indeed consider $M=T^*L$ with double action of $L\times L$. We take $G=L$ as the first copy, $L$ the second copy acting freely on the right. Then $P=L$ and $M=L/L=\{pt\}$ is a point. The equivariant curvature $R_y$, a form on $P$ with value in $\mathfrak l$, is  $R_y(l)=-ly$  (Formula (\ref{cugu}))  so that $\tilde \gamma_{\zeta}(y)$ is the invariant function $\int_{L}e^{-i\langle \zeta,ly\rangle}dl$ and  $\langle \infdex^{\mu}_G([\tilde\gamma_{\zeta}]),f_2\rangle$ (the distribution  Fourier transform  of the function  $\tilde \gamma_{\zeta}(y)$) is $\int_{L} f_2(-l\zeta)dl$. Theorem above gives
$$\langle \infdex_{G\times L}^{ \mu_{G\times L}}(1),f_1f_2\rangle =i^{\dim L}
\int_{L\times \mathfrak l^*}  f_1(\zeta) f_2(-l\zeta) dl d\zeta$$
which is Formula (\ref{infdex1}).

\end{remark}

Let us first write a corollary of this theorem.

\begin{corollary}\label{cor:free}
Let  $f_2$ be a test function on $\mathfrak g^*$. Then the
distribution $f_1\to \langle \infdex_{G\times
L}^{  \mu_{G\times L}}(\pi^*([\gamma])),f_1f_2\rangle$ on $\mathfrak l^*$ is a
smooth density $D(\zeta)d\zeta$. The value of $D$ at $0$ is equal
to
 $i^{\dim L}{\rm vol}(L,dl) \langle
\infdex^{\mu}_G([\gamma]),f_2\rangle.$
\end{corollary}

We now prove Theorem \ref{theo:free}.

\begin{proof}

Denote by   $\eta:P\times
\mathfrak l^*\to P$ the projection $\eta:(p,\zeta)\mapsto  p ,  $    and  set $\xi=\pi\circ \eta: N\to M,\ \xi(p,\zeta):=\pi(p),\ p\in P,\zeta\in \mathfrak l^*$.

Let $\gamma(y)$ be a compactly supported $G$--equivariant form on $M$
  representative of $[\gamma]$.
Any $G\times L$--equivariant form $\psi$  with compact support on
$N=P\times \mathfrak l^*$  which restricted to $P$ coincides with $\pi^*\gamma$  can be
taken as a representative for the cohomology class  $\pi ^*[\gamma]\in
H_{G\times L,c}^*(N^0).$

In order to construct $\psi$,  take an  $ L$-invariant  function $\rho:\mathfrak l^*\to \mathbb R$ supported near zero and such
that $\rho$ equals 1 on a neighborhood of $0$ and define the form
$\psi$, which is still $L$-invariant and $G$-equivariant by:
\begin{equation}\label{psigamma}\psi(y)(p,\zeta):=
\rho( \zeta )\xi^*\gamma(y).\end{equation}

Recall that  $\overline \sigma$ is the restriction of $\sigma_{N}$ on $P$
and consider  the one form $\eta^*(\overline \sigma)$ on $N=P\times
\mathfrak l^*$,  the pull back of $\overline \sigma$  under the projection $\eta:P\times
\mathfrak l^*\to P$. Let $\omega\in \mathcal A^1(P)\otimes \mathfrak l$
be our connection form.
 Then \begin{lemma}\label{acfo}
$\langle \omega,\zeta\rangle $ is an action form on $N$,
 with moment map for $L$ the second projection.
 Its moment map for $G$ vanishes on $P$.
\end{lemma}

Consider $  \sigma_0 =\sigma_N$ and $\sigma_1=\eta^*(\overline \sigma)+\langle
\omega,\zeta\rangle$ with moment maps $\mu_0,\mu_1$.  \begin{lemma}
The  moment map $\mu_t=t\mu_1+(1-t)\mu_0$  associated to $t\sigma_1+(1-t)\sigma_0$ is such that $\mu_t^{-1}(0)=N^0$ for all $t\in [0,1]$.
\end{lemma}\begin{proof}
This follows from the fact that    the component  under $L$ of these maps is the second projection, so that $\mu_t^{-1}(0)\subset P$ for all $t$ and moreover $\mu_1,\mu_0$ coincide on $P$. Thus $\mu_t^{-1}(0)=P^0=N^0$.
\end{proof}  According to Theorem \ref{dein}, we may thus assume that
$\sigma_N=\eta^*(\overline \sigma)+\langle \omega,\zeta\rangle $ and compute with
this ``normal form" the values of $ \infdex_{G\times L}^{ \mu_{G\times L}}$.

Recall that  $\mu:M\to \mathfrak g^*$  is the moment map relative to $G$  associated to $\sigma_M$. By abuse of notations, we still denote by $\mu$ its pull back by $\pi\eta$ to $N$. This is the moment map associated  to $\eta^*(\overline \sigma)$.

\begin{lemma}
Let  $\Omega:=D\sigma_N, $   for $(x,y)\in\mathfrak
l\oplus \mathfrak g$.  At a point $(p,\zeta)\in P\times \mathfrak l^*$, we have: $$\Omega(x,y)=\langle x,  \zeta\rangle-
\langle \iota_y\omega,\zeta\rangle + \langle y,\mu\rangle +
d\eta^*(\overline \sigma)+ d\langle \omega,\zeta\rangle .$$
\end{lemma}
\begin{proof}By the definition of a connection form (for the action of $L$), we have
$
\langle x,  \zeta\rangle =-\langle  \iota_x\omega,  \zeta\rangle $ so $
\langle x,  \zeta\rangle -
\langle \iota_y\omega,\zeta\rangle $ is the value of the moment map  at $(x,y)$  of $\langle \omega,\zeta\rangle$. As for $\eta^*(\overline \sigma)$, by definition of $P=\mu_L^{-1}(0)$, the part relative to $L$ of its moment map equals to 0.
\end{proof}
 We write
$\Omega(x,y)=\langle x, \zeta\rangle+\Omega'(y)$ with

$$\Omega'(y)= -\langle \iota_y\omega,\zeta\rangle +
\langle y,\mu\rangle +  \eta^*d(\overline \sigma)+  d\langle \omega,\zeta\rangle$$  independent of $x$.
We have
 \begin{equation}
\label{omepr}\Omega'(y)=\eta^*(D\overline \sigma) -\langle \iota_y\omega,\zeta\rangle+ d\langle \omega,\zeta\rangle.
\end{equation}

For $s$ sufficiently large,
$$
 \langle \infdex_{G\times
L}^{ \mu_{G\times L}}(\pi^*([\gamma])),f_1f_2\rangle= I(s)$$ with
$$I(s)=\int_N\int_{\mathfrak g\times \mathfrak l}e^{is\Omega(x,y)}\psi(y)
\hat f_1(x)\hat f_2(y)dxdy.$$

Applying  Fourier inversion
$$\int_{\mathfrak l}e^{is\langle x, \zeta\rangle  }\hat f_1(x)dx=f_1(s\zeta), $$
we obtain that
$$I(s)= \int_N\int_{\mathfrak g}e^{is\Omega'(y)}\psi(y)f_1(s\zeta) \hat
f_2(y)dy $$ where $\psi(y)(p,\zeta)=\rho( \zeta )\xi^*\gamma(y)$
is defined by Formula (\ref{psigamma}).
\medskip

Write the connection form $\omega=\sum_{i=1}^r \omega_i e_i$ on a oriented  basis $\{e_1,\ldots
,e_r\}$ of $\l$, and set $\zeta_i=\langle e_i,\zeta \rangle$ for
$i=1,\ldots, r$.

 We have $\langle \omega,\zeta\rangle
=\sum_{i=1}^r \zeta_i \omega_i $ and thus
\begin{equation}\label{omegazeta}
d\langle \omega,\zeta\rangle =\sum_{i=1}^r\zeta_id \omega_i
+\sum_{i=1}^r d\zeta_i \wedge \omega_i.\end{equation}

Let us now integrate along the fiber $\mathfrak l^*$  of the
projection $\eta:N=P\times \mathfrak l^*\to P$. We thus need to identify the
highest term of  $e^{is\Omega'(y)}$ in the $d\zeta_i$. By
 \eqref{omegazeta},\eqref{omepr}  this highest term equals
 $$  (is)^r d\zeta_1\wedge \omega_1\wedge \cdots \wedge d\zeta_r\wedge  \omega_r= (-1)^{\frac{r(r+1)}{2}} (is)^r   V_{\omega}  \wedge d\zeta  $$   where we set
 $ V_{\omega}:=\omega_1\wedge \omega_2\wedge \cdots\wedge  \omega_r$ and  $d\zeta:=d\zeta_1\wedge\cdots \wedge d\zeta_r$.
   We obtain
$$I(s)= \int_N\int_{\mathfrak g}e^{is\Omega'(y)}\psi(y)f_1(s\zeta) \hat f_2(y)dy $$
$$=(-1)^{\frac{r(r+1)}{2}}  i ^r \int_{P\times \mathfrak g} e^{is D  \overline \sigma}\gamma(y)\hat f_2(y)
\left(\int_{\l^*}  s^{r}e^{-is \langle \iota_y\omega ,
\zeta\rangle } e^{is\langle d\omega,\zeta \rangle } \rho(\zeta)
f_1(s\zeta)   V_{\omega}  \wedge d\zeta
\right)dy.$$ In the integral on $\l^*$, we change $\zeta$ to
$s\zeta$ and obtain
$$(-1)^{\frac{r(r+1)}{2}}  i ^r \int_{P\times \mathfrak g}
e^{is D \overline \sigma }\gamma(y)\hat f_2(y) \left(\int_{\l^*} e^{-i
\langle \iota_y\omega , \zeta\rangle } e^{i\langle d\omega,\zeta
\rangle } \rho(\zeta/s) f_1(\zeta)   V_{\omega}  \wedge d\zeta   \right) dy.$$

On the compact support of $f_1(\zeta)$, if $s$ is sufficiently
large, $\rho(\zeta/s)=1$. Also we may replace $d\omega$ by $R$ as
$R-d\omega=\frac{1}{2}[\omega,\omega]$ is annihilated by wedge
product with $\omega_1\wedge \omega_2\wedge \cdots\wedge \omega_r$
and obtain (for $s$ sufficiently large): \label{before}
$$\langle \infdex_{G\times L}^{\mu_{G\times L}}\psi,f_1f_2\rangle$$
$$=(-1)^{\frac{r(r+1)}{2}}i^r
\int_{N}\int_{\mathfrak g} e^{is \eta^*D\overline \sigma} \gamma(y)\hat f_2(y)
e^{i \langle R_y,\zeta\rangle }   f_1(\zeta)   V_{\omega}  \wedge d\zeta  dy.$$

Now consider the fibration $N\to M\times \l^*$ with fiber $L$. On each fiber, the form $V_{\omega} =\omega_1\wedge \omega_2\wedge \cdots\wedge
\omega_r$ induces an orientation and restricts to a  Haar measure $dl$ on $L$.
Let us now integrate
over the fiber.
 Recall that $\sigma_M$ denotes the action form on $M$. Let  $\Omega_M:=D\sigma_M$, we have $\pi^*\sigma_M=\bar\sigma,   \eta^*D\overline \sigma=\eta^*\pi^*\Omega_M$. By Formula \eqref{liofo1} recalling that $r=\dim L$  and
 using Formula \eqref{thzr},  we finally obtain that $I(s)$ is equal to
$$i^{\dim L}\int_{\mathfrak l^*}\left(\int_{M}\int_{ \mathfrak g}
e^{is \Omega_M(y)} \gamma(y)\theta_{\zeta}(R_y)\hat f_2(y)dy
\right)f_1(\zeta)d\zeta.$$

Remark that when $\zeta$ varies in the  compact support
of $f_1$, and over a compact subset $K$ of $M$, the Fourier
transform (in $y$) of $\theta_{\zeta}(R_y)$ stays supported on a
fixed compact subset of $\mathfrak g^*$. Indeed, using Formula
(\ref{fou}), we see that the Fourier transform of
$\theta_\zeta(R_y)$ is supported on the compact subset
$h(\pi^{-1}K,L\zeta)$.
 By Remark \ref{rem:stableinfdex}, for $s>>s_0$
\begin{equation}
\label{ryz}\int_{M\times
\mathfrak g} e^{is \Omega_M(y)} \gamma(y)\theta_{\zeta}(R_y)\hat
f_2(y)dy=\infdex^{\mu}_G([\tilde\gamma_{\zeta}],f_2\rangle
\end{equation} for any
$\zeta$ in the support of $f_1$.

 Thus we obtain our claim.
\end{proof}

%
%
%
%
%
%
%
%
%
%
%
%
%
%
%
%
%
%
%
%
%
%
%
%
%
%
%
Another important particular case  of the free action property is
when $G$ is trivial. We then have $y=0$ in all the steps of the
proof of Theorem \ref{theo:free}. We summarize the result that we
obtain in this particular  case of Theorem \ref{theo:free}.
 Let $N$ be an oriented
$L$-manifold with action form, and assume that  the group  $L$
acts freely on $N$ and that $0$ is a regular value of $\mu_L$. Let
$M=N^0/L$ and let $[\gamma]\in H_{G,c}^*(N^0)=H^*_c(M)$.

Let $R$ be the curvature of the fibration $N^0\to M$.
  For any
$\zeta\in \mathfrak l^*$, we consider the closed differential form
on $M$ given by
\begin{equation}\label{ayzero}
\theta_\zeta(R)=\int_{L}
  e^{i\langle R,l\zeta\rangle} dl.
\end{equation}
Here, as $R$ is a $\mathfrak l$ valued two form, $\theta_\zeta(R)$
is a polynomial function of $\zeta.$

Then we obtain
\begin{proposition}

 The distribution $\infdex_L^{\mu_L}([\gamma])$ is a polynomial
density on $\mathfrak \l^*$. More precisely
 $$\langle \infdex_L^{\mu_L}([\gamma]),f_1\rangle =i^{\dim L}
 \int_{\mathfrak l^*} \left(\int_M \gamma
\theta_\zeta(R)\right) f_1(\zeta) d\zeta.$$
\end{proposition}

In particular  the value of $\infdex_L^{\mu_L}([\gamma])$  at $0$
is well defined and computes the integral on the reduced space
$\mu_L^{-1}(0)/L$ of the compactly supported cohomology class
associated to $[\gamma]$. This is    essentially Witten
localization formula \cite{wit},\cite{jef-kir}.

\subsection{Extension of the properties of the infinitesimal index}

We have extended the definition of the infinitesimal index to
$\mathcal H_{G,c}^{\infty,m}(M^0)$.
Analyzing the proofs of the  properties  {\em  locality, product, the map $i_!$}, we see that   these properties hold for the infinitesimal index map on  $\mathcal H_{G,c}^{\infty,m}(M^0)$. The proofs for the restriction property and the free action extend, provided we are in the situation of Remark \ref {rem:stableinfdex}:  we consider the infinitesimal index on classes $[\alpha] \in H_{G,c}^{\infty,m}(M^0)$ such that the Fourier transform of $\alpha(x)$ is a distribution with compact support on $\mathfrak g^*$, so that the infinitesimal index stabilizes for $s$ large.
This will be always the situation in the applications to index formulae.

\section{Some consequences of the functorial properties of the
infinitesimal index}

We list here some corollaries of the functorial properties:
excision, product, restriction, push-forward, free action proved
in the preceding section.

\subsection{Diagonal action and convolution}

Consider two $G$  action manifolds $M_1,M_2$ with moment maps
$\mu_1,\mu_2$ with zeroes $M_1^0,M_2^0$. Let $\Delta$  be the diagonal subgroup. The moment map
for $\Delta$ is $\mu_1+\mu_2$.

 Let us assume that $(M_1\times
M_2)_{\Delta}^0=M_1^0\times M_2^0$.
If $\alpha\in H_{G,c}^*(M^0_1)$ and $\beta\in H_{G,c}^*(M^0_2)$,
we may apply the product property (Proposition \ref{product}) and the restriction property (Theorem \ref{restriction}). As     the
restriction map is such that $r^*f(\xi_1,\xi_2)=f(\xi_1+\xi_2)$ ($\xi_1,\xi_2\in \mathfrak g^*$), we obtain the following proposition.

 \begin{proposition}\label{convol}
 Under the hypothesis
$(M_1\times M_2)_{\Delta}^0=M_1^0\times M_2^0$,
 the infinitesimal index $\infdex_{\Delta}^{\mu_1+\mu_2}(\alpha_1\wedge \alpha_2)$
 is the convolution product
$\infdex^{\mu_1}_G(\alpha_1)*\infdex^{\mu_2}_G(\alpha_2)$ of  the distributions
$\infdex^{\mu_1}_G(\alpha_1)$ and $\infdex^{\mu_2}_G(\alpha_2)$.
\end{proposition}

Let us give an important example of this situation.

Let $M_X$ be a complex representation space for the action of a
torus $G$, where $X=[a_1,a_2,\ldots, a_m]$ is a list of nonzero weights
$a_i \in \hat G\subset \mathfrak g^*$.
We assume that $X$ spans a pointed cone
in $\mathfrak g^*$. Recall the definition of the {\it multivariate
spline $T_X$}, it is a tempered distribution defined by:
\begin{equation}\label{multiva}
\langle T_X\,|\,f\rangle = \int_0^\infty\dots\int_0^\infty
f(\sum_{i=1}^mt_i a_i)dt_1\cdots dt_m.
\end{equation}

Let us consider on $M_X=\mathbb C^m$ the action form  such that
$\mu(z_1,\ldots, z_m)=\sum_{i}\frac{|z_i|^2}{2}a_i$. Then
$M_X^0=\{0\}$ and the class $1$ is a class in $H_{G,c}^*(M_X^0)$.
Using our computation in Example \ref{Heaviside} of
$\infdex_G^{\mu}(1)$ in the case of $\mathbb R^2=\mathbb C$, we
obtain the following formula.

\begin{proposition}

$$\infdex_G^{\mu}(1)=(2\pi i)^mT_X.$$

\end{proposition}

We will use this calculation in \cite{dpv2} to identify
$H_{G,c}^*((T^*M_X)^0)$ to a space of spline distributions on
$\mathfrak g^*$.

\bigskip

Another example that we will use in Subsection \ref{maxtori}
  is
the case were one of the action forms, say $\sigma_1$, is equal to
$0$, so that $\mu_1=0$ and $\mu$ is the pullback of $\mu_2$.
Then $$(M_1\times M_2)_\Delta^0=M_1\times
M_2^0.$$

In this case, the space  $H_{G,c}^*(M_1^0)$ is simply
$H_{G,c}^*(M_1)$ and $\int_{M_1}\alpha_1(x)$ is a polynomial
function of $x\in\mathfrak g$. Thus $\infdex^{0}_G(\alpha_1)$, the Fourier transform, is a
distribution of support $0$ on $\mathfrak g^*$.
\begin{corollary}\label{preinf}
$${\rm infdex}_{\Delta}^{\mu}[\alpha_1 \times\alpha_2]=
\infdex^{0}_G(\alpha_1)*\infdex^{\mu_2}_G(\alpha_2).$$
\end{corollary}

\subsection{Induction of distributions}\label{induction0} Assume that $L\subset G$ is a subgroup, let $\mathfrak l\subset \mathfrak g$ be the corresponding Lie  algebras.
 Choose  Lebesgue measures on $\g$, and  $\mathfrak l$  by fixing translation invariant top differential forms. This determines dual measures and forms on  $\mathfrak g^*,$ $\mathfrak l^*$ and a Haar measure $dg$ on $G$.
If $p$ is the restriction map $\mathfrak \g^*\to \mathfrak l^*$,
we let $p_*$ be the integration over the fiber (with respect to the chosen forms and orientations). It sends a test
function on $\mathfrak g^*$ to a test function on $\mathfrak l^*$.
Let
\begin{equation}
\label{Aver}A(f)(\xi)=\int_{G} f(g\xi)dg.
\end{equation}
The operator $A$ transform a test function on
$\mathfrak g^*$ to an invariant test function on $\mathfrak g^*$.

\begin{definition}\label{inductiond}
For a distribution $V$    on $\mathfrak l^*$, we define the
$G$-invariant distribution  ${\rm Ind}_{\mathfrak l^*}^{\mathfrak
g^*} V$  on $\mathfrak g^*$ by
 $$\langle {\rm Ind}_{\mathfrak l^*}^{\mathfrak g^*} V,f\rangle={\rm vol}(L,dl)^{-1}
 \langle V ,  p_*(A(f)) \rangle,$$
 $f$  being a  test function on $\mathfrak g^*$.

\end{definition}

It is easy that $  {\rm Ind}_{\mathfrak l^*}^{\mathfrak g^*} V$ is independent of the choices of measures.

\subsection{Induction of action manifolds}\label{induction}
Assume that $L\subset G$ is a subgroup. Take $M$ a $L$
manifold with action form $\sigma$ and moment map $\mu_L$.

Consider  $T^*G$ as a $G \times L$  action manifold where $G$ acts on the
left and $L$ on the right, and  the action form $\omega$  is the canonical one form  on $T^*G$.

\subsubsection{The induced action
manifold}
Set $N:=T^*G\times M$ and  $p_1,p_2$ be
the first and second projection of this product manifold. We consider
the action form $\psi=p_1^*\omega+p_2^*\sigma$ on $N$, and denote by
$\tilde \mu_{G\times L}={\tilde \mu}_G\oplus {\tilde \mu}_L$ the
corresponding moment map.\smallskip

Let us
   trivialize  $T^* G=G\times
\mathfrak g^*$ using left trivialization (\ref{mmm}), so that we identify  $N=G\times
\mathfrak g^*\times M$.  According to  Formula \eqref{mmm}, if $(g,\xi,m)\in N$ we have:

\begin{equation}\label{mmm2} {\tilde \mu}_G(g,\xi,m)=- g\xi:=-g\xi,\quad {\tilde \mu}_L
(g,\xi,m)= -\xi |_{\mathfrak l}+\mu_L(m).
\end{equation}

We denote by $N^0$ the  zero fiber  of the moment map
$\tilde \mu_{G\times L}$  for    $G\times L$,  by $M^0$ the zero fiber of
the moment map $\mu_L$ on $M$  for   $ L.$  \begin{lemma}\label{ennezero}
We have $N^0=   G\times M^0$.
\end{lemma}\begin{proof}
From Formula \eqref{mmm2} the  set of points of $N$ where ${\tilde \mu}_G=0$ is $G\times M$, and on these points we have  ${\tilde \mu}_L(g,m)={  \mu}_L(m).$
\end{proof}

\begin{lemma}\label{zepp}\begin{enumerate}
\item  If we take the zero fiber of
${\tilde \mu}_L$, we obtain the manifold
\begin{equation}\label{eqP}
P:={\tilde \mu}_L^{-1}(0)=\{(g,\xi,m);
g\in G, \xi\in \mathfrak \g^*,m\in M; \xi|_{\mathfrak
l}=\mu_L(m)\}.
\end{equation}
\item  $0$ is a regular value for  the moment map ${\tilde \mu}_L$.
\end{enumerate}

\end{lemma}
\begin{proof}
The first statement is immediate from Formula \eqref{mmm2}. As for the second, fix $g,m$,  the map $\tau:\mathfrak g^*\to\mathfrak l^*$ given by $\tau:\xi\mapsto {\tilde \mu}_L
(g,\xi,m)= -\xi |_{\mathfrak l}+\mu_L(m)$ has clearly surjective differential  for all $\xi$  hence the claim. \end{proof}

We are thus in the situation of Subsection \ref{sub:fre}.
The manifold $N$ is a $G\times  L$  manifold, $L$ acts freely on $N$ and $0$ is a regular value of the moment map  $\tilde \mu_L$  for $L$.
Consider the manifold $\mathcal M=P/L$.
Applying Proposition \ref{prop:sigmabar}, we see
 \begin{lemma}
The quotient $\mathcal M=P/L$  is a $G$-manifold.
The action form  on $N$  restricted to $P$ descends to  $\mathcal M$.

The induced   moment map $\mu_G:P/L\to
\g^*$ is obtained by quotient  from the moment map ${\tilde \mu}_G:(g,\xi,m)\to g\xi$ on $P$.
\end{lemma}

 \begin{definition}
We will say that $\mathcal M$ is the {\em induced action
manifold}.
\end{definition}
By Lemma \ref{ennezero}, the closed set  $N^0$, the  zero fiber  of the moment map
$\tilde \mu_{G\times L}$, equals  $G\times M^0$ and it is contained in $P$. Since, by definition, on $P=\tilde \mu_{L}^{-1}(0)$  the moment map  $\tilde \mu_{  L}$ equals 0,   we have that on $P$  the moment map  $\tilde \mu_{G\times L}$ equals  $\tilde \mu_{G }$. Therefore we obtain the \begin{lemma}\label{zefi} Under the inclusions $N^0\subset  P,\ N^0/L\subset  P/L$,
the zero
fiber $\mathcal M^0_G\subset \mathcal M$ of the moment map $\mu_G$ is identified with $N^0/L=G\times_L
M^0$.
\end{lemma} Denote by $p_1,p_2$ the two projections of  $ N^0=G\times  M^0$ on its factors.  Denote by  $\pi :G\times  M^0=N^0\to N^0/L= G\times_L M^0 $ the quotient map.

Thus we get isomorphisms
$$\begin{CD} H^*_{L,c}(M^0)@>p_2^*>> H^*_{G\times L,c}(N^0)@<\pi^*<<H^*_{G,c}(G\times_L M^0 )\end{CD}.$$

We set $j=\pi^{*-1}p_2^*$: \begin{equation} \label{lai}\begin{CD} j:
H^*_{L,c}(M^0)@>p_2^*>> H^*_{G\times L,c}(N^0)@>
(\pi^*)^{-1}>>H^*_{G,c}(\mathcal M^0_G).\end{CD}
\end{equation}

\begin{remark}\label{formulaj}
As in the usual case (see \cite{duf-kum-ver}, page 33), the isomorphism $j^{-1}$ can be described as follows. Let $\gamma(y)$, with $y\in \mathfrak g$, be  an equivariant form on $P/L=\mathcal M$ representing $[\gamma]\in H_{G,c}^*(\mathcal M_G^0)= H_{G,c}^*(G\times_L M^0)$. We restrict $\gamma$ to the $L$ invariant  submanifold $M$ embedded in $\mathcal M$ by $m\mapsto (e,\mu_L(m),m)$ and obtain an $L$-equivariant form on $M$. We can represent $j^{-1}[\gamma]$  by $\gamma(x)|_M$ with $x\in \mathfrak l$.

\end{remark}

\subsubsection{The induction formula for {\em infdex}}\label{inductioninfdex}
 Given a class $[\alpha]\in H^*_{L,c}(M^0)$, our goal is to compare
 $ \infdex_L^{\mu_L}( [\alpha])$ and
 $\infdex_G^{\mu_G}(j( [\alpha]))$,
 the first being a distribution on $\mathfrak l^*$
 and the second one on $\mathfrak g^*.$
  We shall show that $ \infdex_G^{\mu_G}(j( [\alpha]))$
   {\em is induced} by
  $\infdex_L^{\mu_L}( [\alpha])$, according to Definition \ref{inductiond}.
  \begin{theorem}\label{laind} Let $[\alpha]\in H^*_{L,c}(M^0)$, then
\begin{equation}\label{indi}
\infdex_{G}^{\mu_G} (j [\alpha])=i^{\dim G-\dim L}{\rm Ind}^{ \mathfrak g^* }_{
\mathfrak l^*}(\infdex^{\mu_L}_{L}([\alpha])).
\end{equation}
\end{theorem}

\begin{proof}
Consider   the form $\gamma:=1\times \alpha$ on $G\times M$, where
$\alpha$ is a representative of $[\alpha]$.
   By   definition  $j={\pi^*}^{-1}p_2^*$ and we see that
        $[\gamma]=p_2^*[\alpha]=\pi^*j[\alpha]$.

Consider the $G\times L$ manifold  $N= T^*G\times M$. To this manifold we can apply  Corollary \ref{cor:free}. Let $f_1$ be
a  variable test function on $\mathfrak l^*$ and $f_2$ be a given test function
on $\mathfrak g^*.$  The distribution
$f_1\to \langle \infdex_{G\times L}^{\tilde \mu_{G\times
L}}([\gamma]),f_2 f_1\rangle$ is given by a smooth density $D(\zeta) d\zeta$ on
$\mathfrak l^*$, and the value  $D(0)$ equals
$i^{\dim L}{\rm vol}(L,dl)\langle \infdex_{G}^{\mu_G} (j [\alpha]),f_2\rangle .$

\smallskip

Let us compute   $\langle \infdex_{G\times
L}^{\tilde \mu_{G\times L}}([\gamma]),f_2 f_1\rangle $ using the fact
that $\gamma$ is the external product $1\times \alpha$. We
consider the product manifold $T^*G \times M$ provided with the
action of $G_1 \times G_2$ where $G_1= G\times G$ acts by left and right action on $T^*G$ and $G_2=L$ acts on $M$.

Consider next the embedding of
$G\times L$ as   subgroup of
$G \times G \times L$ by  $\{((g,l),l),g\in G,l\in L\}$ .   Denote by
$$s:\mathfrak g\oplus \mathfrak l\to \mathfrak g\oplus \mathfrak g\oplus \mathfrak l,\ (a,b)\mapsto (a,b,b)$$  the inclusion of Lie algebras.  Denote by $p:\mathfrak g^*\to \mathfrak l^*$     the  restriction map.  Then for
 $\zeta\in\mathfrak l^* $ and $(\xi_1,\xi_2)\in
\mathfrak g^*\oplus \mathfrak g^*$ the restriction map $R$ associated to   the inclusion $s$ is:
  $$R:=s^*:\mathfrak g^*\oplus \mathfrak g^* \oplus \mathfrak l^*  \to \mathfrak g^*\oplus \mathfrak
l^*,\ (\xi_1,\xi_2,\zeta)\mapsto (\xi_1,\zeta+p(\xi_2)).$$ Remark that our given action form on $N$  is $G \times G \times L$
invariant and that $N^0=G\times M^0$ is also the set of zeroes of the moment map  $\mu$ for  the group $G \times G \times L$.

 In order to compute $\infdex_{G\times
L}^{\tilde \mu_{G\times L}}([\gamma])$ we may thus apply first the external product property
(Proposition \ref{product}) and then the restriction property
(Proposition \ref{restriction})  obtaining:$$\infdex_{G\times L}^{\tilde \mu_{G\times L}}(1\times \alpha)=
R_*(\infdex_{G\times G}^{\mu}(1)\otimes
\infdex_L^{\mu_L}([\alpha])).$$

We now make this formula more explicit.
  Let $f_1$ be a
test function on $\mathfrak l^*$ and  $f_2$   a test function on
$\mathfrak g^*$.  Using Formula  \eqref{plos}$$\langle R_*(\infdex_{G\times G}^{\mu}(1)\otimes
\infdex_L^{\mu_L}([\alpha])),f_1f_2\rangle$$$$
=\lim_{T\to \infty}\langle \infdex_{G\times G}^{\mu}(1)\otimes
\infdex_L^{\mu_L}([\alpha]),R^*(f_1 f_2)\chi_T\rangle$$ The function $R^*(f_1 f_2)(\xi_1,\xi_2,\zeta)$ is the
function $f_1(\zeta+p(\xi_2)) f_2(\xi_1).$
Using the formula for  $\infdex_{G\times G}^{\mu}(1)$ for $T^*G$
of Proposition \ref{GG}, we obtain
$$\lim_{T\to \infty} \langle \infdex_{G\times G}^{\mu}(1)\otimes
\infdex_L^{\mu_L}([\alpha]),f_1(\zeta+p(\xi_2)) f_2(\xi_1)\chi_T\rangle$$$$=i^{\dim G}\langle \infdex_L^{\mu_L}([\alpha]), q(f_1,f_2)\rangle$$
with ($A$ is defined in \eqref{Aver}):
$$q(f_1,f_2)(\zeta)=\int_{ \mathfrak g^*}\int_G  f_1(\zeta+p(\xi))
f_2(-g\xi) dg d\xi=\int_{\mathfrak g^*} f_1(\zeta+p(\xi)) Af_2(-\xi)
d\xi.$$ Integrating first on the fiber $p:\g^*\to \mathfrak l^*$,
then on $\mathfrak l^*$, we see that $$q(f_1,f_2)(\zeta)=f_1*(p_*(Af_2))(\zeta)$$
where $u*v$ is the convolution product of test functions on
$\mathfrak l^*$.

 Then we obtain
$$\langle \infdex_{G\times L}^{\tilde \mu_{G\times L}}([\gamma]),f_1 f_2\rangle =i^{\dim G}
 \langle \infdex_{L}^{\mu_L}([\alpha]),f_1*(p_*(Af_2))(\zeta)\rangle .$$
This is a smooth density with respect to $\zeta\in \mathfrak l^*$,
and if $f_1$  tends to $\delta_0(\zeta)$, then
$\langle\infdex^{\tilde \mu_{G\times L}}_{G\times L}([\gamma]),f_1 f_2\rangle$ tends
to $$i^{\dim G}\langle \infdex_{ L}^{\mu_L}([\alpha]),p_*(A f_2)(\zeta)\rangle=i^{\dim G}
{\rm vol}(L,dl)\langle {\rm Ind}_{\mathfrak l^*}^{\mathfrak g^*} \infdex_{
L}([\alpha]), f_2(\zeta)\rangle.$$ We thus obtain the wanted formula \eqref{indi}.
\end{proof}

 \subsection{Maximal tori}\label{maxtori}

As usual, let $M$ be a $G$-manifold with a $G$-invariant action form $\sigma$.
Let $T\subset G$ be a maximal
torus.
We show next how to reduce the calculation of the infinitesimal index map for $G$ to the calculation of the infinitesimal index map for $T$.
Our construction is very similar to the construction of the map $K_G(T^*_G N)\to K_T(T^*_TN)$ at the level of $K$-theory given in \cite{At}.

Associated to $\sigma$,  we have the  moment maps $\nu_G:M\to \mathfrak g^*$ and
$\nu_T=p\circ \nu_G:M\to\mathfrak t^*$, with $p:\mathfrak g^*\to
\mathfrak t^*$   the restriction map.

Consider $M$ as a $T$-manifold, and consider $N=T^*G\times M$,
provided, as in Subsection
\ref{induction} (here the group $L$ is $T$), with action form $\psi=p_1^*\omega+p_2^*\sigma$ and the action of $G\times T$: the group $G$ acts on $T^*G$ by
  left action, and trivially on $M$,
 the group  $T$ acts on $G$ by right action and acts on $M$.  We denote by $\tilde \mu_{G\times T}=\tilde \mu_G\oplus \tilde \mu_T$ the corresponding moment map.

Recall, by Formula \eqref{eqP}, that
  $$P=\tilde\mu_T^{-1}(0)=\{(g,\xi,m); g\in G, \xi\in \mathfrak \g^*,m\in M;
\xi|_{\mathfrak t}=\nu_T(m)\}$$ is a $G\times T$ manifold
on which $G$ acts by  $g_0\cdot (g,\xi,m)=(g_0g,\xi,m)$, for $g_0\in G$, $(g,\xi,m)\in P$ and $T$ acts by $t\cdot (g,\xi,m)=(g t^{-1},t\xi,tm)$.

We then consider  $\mathcal M:= P/T$,
with moment map $\mu_G([g,\xi,m])=g\xi $  \eqref{mmm2}.
Recall that
$\mathcal M^0_G$
 is isomorphic to $G\times_T M_T^0$ embedded in $P/T$ by $[g,0,m]$.

\bigskip

For $[\alpha]\in H^*_{G,c}(M^0_G)$, we want to  produce an element
 $r([\alpha]) \in  H^*_{G,c}(\mathcal  M_G^0)=H^*_{G,c}(G\times_T M_T^0)$
which has the same infinitesimal index as $[\alpha]$.

\bigskip

\begin{proposition}
We can embed $G\times M$ in $P$ by the map
$$\gamma(g,m)=(g,\nu_G(g^{-1}m),g^{-1}m).$$
The map $\gamma$ is $G\times T$  equivariant, where $G$ acts on $G\times M$ by diagonal  action (left  on $G$) while $T$ acts by the right action on $G$ and not on $M$.\end{proposition}
\begin{proof}
First $(g,\nu_G(g^{-1}m),g^{-1}m)\in P$  since $\nu_G(g^{-1}m)|_{\mathfrak t^*}= \nu_T(g^{-1}m).$   Next $\gamma(hg,hm)=  (hg,\nu_G(g^{-1}m),g^{-1}m)$ and $$\gamma( gt^{-1}, m)=  ( gt^{-1},\nu_G(t g^{-1}m),t g^{-1}m)=  ( gt^{-1},t \nu_G(  g^{-1}m),t g^{-1}m).$$
\end{proof}
\begin{corollary}\label{gthr} The map $\gamma$ induces, modulo the action of $T$, an embedding still denoted by $\gamma:G/T\times M\hookrightarrow \mathcal M=P/T.$
Thus the manifold $G/T\times M$, with diagonal $G$-action is identified to a $G$-invariant submanifold of $\mathcal M $.
\end{corollary}
 In fact more is true. Let $q:\mathcal M\to  G/T\times M $ be the projection given by $q(g,\xi,m)=(gT,gm)$.
 Let $\g^*=\mathfrak t^*\oplus \mathfrak t^{\perp}$ be the  canonical
$T$-invariant decomposition of $\g^*$.
 Then  we claim that
   \begin{proposition}\label{vebP} $q\gamma$ is the identity and
$q:\mathcal M\to  G/T\times M $ is a vector bundle with fiber $\mathfrak t^\perp$.
\end{proposition}\begin{proof}
The first claim comes from the definitions.  As for the second, we may   identify $P$ with $G\times M\times   \mathfrak t^\perp$   by the map
$$P\to G\times M\times   \mathfrak t^\perp,\quad (g,\xi,m)\mapsto   (g,m,\xi-\nu_T(m)). $$
\end{proof}
\begin{lemma}
 The restriction  of the moment map $\mu_G$ on $\mathcal M$ to  $G/T\times M$ is just $(gT,m)\mapsto \nu_G(m)$ with zeroes $G/T \times M_G^0$.
\end{lemma}
\begin{proof}
We have $\mu_G(g,\xi,m)=g\xi$  by the previous discussion.  An element $(g , m)$ corresponds to the triple $(g,\nu_G(g^{-1}m),g^{-1}m)$,  so the claim follows since $\nu_G$ is $G$--equivariant.
\end{proof}
We now apply the construction $\gamma_!$ of Subsection \ref{push} to the manifold
$G/T\times M$ embedded by $\gamma$
in $\mathcal M$.

Recall that $G/T$ is a even dimensional manifold.
Take  an equivariant form   $\beta$ on  $G/T$ with class
\begin{equation}
\label{beta}[\beta]=(-1)^{\frac{1}{2}\dim G/T}\,\frac{e(G/T)}{|W|}
\end{equation}
where $W$ is the Weyl group and $e(G/T)$ is the equivariant Euler
class. Notice that since $|W|$ equals the Euler characteristic of
$G/T$,  $\int_{G/T}[\beta]$ is equal to $(-1)^{\frac{1}{2}\dim G/T}.$

Thus, by Theorem \ref{theowit1}, the infinitesimal index of $[\beta]$ is just the $\delta$--function  on $\mathfrak g^*$.
 Let $[\alpha]\in H_{G,c}^*(M_G^0)$. We then  construct the element $[\beta\wedge \alpha]$ in the
compactly supported equivariant cohomology  $$[\beta\wedge \alpha]\in H_{G,c}^*((G/T\times
M)^0_G) =H_{G,c}^*(G/T\times M^0_G).$$
 \begin{lemma}\label{besae} The infinitesimal index of
$[\beta\wedge \alpha]$ is   equal to $(-1)^{\frac{1}{2}\dim G/T}$ times the infinitesimal index of
$[\alpha]$.
\end{lemma}\label{beae}
\begin{proof}
Apply  Corollary \ref{preinf}.
\end{proof}

Under  the embedding  $\gamma:G/T\times M \hookrightarrow
\mathcal M$ of  action manifolds (cf. \ref{gthr}), by Theorem \ref{ilsh}, we
have now a homomorphism
$$\gamma_!:H_{G,c}^*(G/T\times M^0_G)\to H_{G,c}^*(\mathcal M^0_G)$$
preserving infdex.

We define  \begin{equation}
\label{ralfa}r([\alpha]):=
\gamma_!([\beta\wedge\alpha])\in H_{G,c}^*(\mathcal M^0_G).
\end{equation}

We then have, combining Lemma \ref{besae} with Theorem \ref{ilsh}
\begin{equation}
 {\rm infdex}_G^{\nu_G}([\alpha])={\rm infdex}_G^{\mu_G} (r[\alpha]).
\end{equation}
On the other hand, we have the isomorphism
$$j: H_{T,c}^*(
M^0_T)\to H_{G,c}^*(G \times_T M^0_T)$$ and we have shown in
Theorem \ref{indi} that
$$\infdex_{G}^{\mu_G} (j [\theta])=(-1)^{\frac{1}{2}\dim G/T}
{\rm Ind}^{ \mathfrak g^* }_{ \mathfrak
t^*}\infdex^{\nu_T}_{T}([\theta])$$ for any $[\theta]\in
H_{T,c}^*( M^0_T)$.

 We deduce
\begin{theorem}\label{maxtor} Take the commutative diagram
\begin{equation}
\begin{CD}
H_{G,c}^*(  M^0_G)@> r>> H_{G,c}^*(G \times_T
M^0_T)@>j^{-1}>>H_{T,c}^*(
M^0_T)\\@V\rm{infdex}VV@V\rm{infdex}VV@V\rm{infdex}VV\\\mathcal D'
(\mathfrak g^*)^G@>id>> \mathcal D' (\mathfrak g^*)^G@<{\rm Ind}^{ \mathfrak
g^* }_{ \mathfrak t^*}<<\mathcal D'(\mathfrak t^*).
\end{CD}
\end{equation}
The element $[\lambda]:=j^{-1}r([\alpha])\in  H_{T,c}^*( M^0_T)$ is such that
\begin{equation}
{\rm infdex}_G^{\nu_G}([\alpha])={\rm Ind}^{ \mathfrak g^* }_{ \mathfrak
t^*}\infdex^{\nu_T}_{T}( j^{-1}r([\alpha]) )
\end{equation}
\end{theorem}

\bigskip

Let us finally give an explicit  formula for the element
$[\lambda]=j^{-1}r([\alpha])\in  H_{T,c}^*( M^0_T)$ corresponding to
$[\alpha]\in H^*_{G,c}(M^0_G)$.

Let ${\rm Pf}(x)= \det_{\mathfrak t^{\perp}}^{1/2}(x)$ be the
Pfaffian associated to the action of $x\in \mathfrak t$ in the
oriented orthogonal space $\mathfrak t^{\perp}$.

We need the
\begin{proposition}
\label{resaun}The restriction of the form $\beta(x)$ at the point $e\in G/T$ is  the polynomial  $|W|^{-1}(2\pi)^{-{\frac{1}{2}\dim G/T}}{\rm Pf}(x)$.  \end{proposition}
\begin{proof}
By construction, the equivariant Euler class is the restriction to $G/T$ of the Thom class of the tangent bundle.  The fiber of the tangent bundle at the $T$ fixed point $e$ is isomorphic to $\mathfrak t^{\perp}$. Thus this class restricts at the fixed point $e$ as  $(-2\pi)^{-{\frac{1}{2}\dim G/T}}{\rm Pf}(x)$
(\cite{Mathai-Quillen}, see \cite{ber-get-ver}, Theorem 7.41, \cite{par-ver1}).
 \end{proof}

Recall the decomposition $\mathfrak g^*=\mathfrak t^*\oplus \mathfrak t^{\perp}$.
Let us consider the map $\nu_{\perp}:M\to
\mathfrak t^{\perp}$ which is uniquely defined by the identity
 $\nu_G=\nu_T\oplus \nu_{\perp}.$
Then $\nu_T^{-1}(0)\cap \nu_{\perp}^{-1}(0)=\nu_G^{-1}(0)$.

Denote by $\tau_0$ the $T$-equivariant Thom class  of the
 embedding $0\to \mathfrak t^{\perp}$, a compactly supported  equivariant class on $\mathfrak t^{\perp}$. Then
 $\tau_{\perp}:=\nu_{\perp}^*\tau_0$ is a closed equivariant class on
 $M$ supported on a small neighborhood $A$ of  $\nu_{\perp}^{-1}(0)$.
It follows that
\begin{lemma}
 If $[\alpha]\in H_{G,c}^*(M^0_G)$, we can choose $\tau_0$  so that the class
$\tau_{\perp}\wedge \alpha$ defines a class in $H_{T,c}^*(M^0_T)$.
\end{lemma}
\begin{proof}  Let $K\subset M\setminus M^0_G$ be the support of $D\alpha$, then
$D(\tau_{\perp}\wedge \alpha)=\tau_{\perp}\wedge D \alpha$  is supported in  $A\cap K$.  Since $\emptyset= K\cap M^0_G=K\cap M^0_T \cap \nu_{\perp}^{-1}(0)$, we can choose $\tau_0$ so that $A\cap K\cap M^0_T=\emptyset$.
\end{proof}

By Remark \ref{formulaj}, an equivariant form representing $j^{-1}r([\alpha])$ is  the restriction to $M=\{(e,0,m), m\in M\}$ of $r(\alpha)(x)$, when $x\in \mathfrak t$. We still denote it by $j^{-1}(r(\alpha))(x)$.
\begin{theorem}\label{form:maxtor}
We can choose the Thom classes so that  $$j^{-1}(r(\alpha))(x)=|W|^{-1}(2\pi)^{-{\frac{1}{2}\dim G/T}}{\rm Pf}(x) \alpha(x)\wedge\tau_{\perp}(x).$$

\end{theorem}
\begin{proof}

  Let $\tau_{G/T\times M}$ be a Thom class of the bundle $q:\mathcal M\to   G/T\times M $ (Proposition \ref{vebP}).
 Then, by the $\gamma_!$ construction,  the associated equivariant  form on $\mathcal M$ which we denoted by $r(\alpha)$,   is $q^*(\beta\wedge \alpha)\wedge \tau_{G/T\times M}$.

 Now the bundle $q:\mathcal M \to G/T\times M$
 is trivial over $e\times M$ and isomorphic to $\mathfrak t^{\perp}\times M$ by $(\xi,m)\mapsto (e,\xi+\nu_{G}(m),m)$.

 The restriction of the Thom class  $\tau_{G/T\times M}$ gives a Thom class for this trivial bundle. We can then assume that the restriction of  $\tau_{G/T\times M}$ is $\tau_0(\xi)$.

 As $(e,M)$ is embedded by $\xi=\nu_{\perp}(m)$, we obtain our Theorem from Proposition \ref{resaun}.

\end{proof}

\appendix
\section{  Equivariant cohomology with compact supports}

\subsection{Compact supports} We are going to assume in this appendix  that all our spaces are locally compact and paracompact and we are going to work with Alexander-Spanier cohomology groups both ordinary and with compact supports, and with real coefficients. We shall denote them by   $H^*$ or, if we take compact supports, by $H_c^*$. $H^*$ is a cohomology theory  on spaces or pairs of spaces deduced from a functorial cochain complex $\mathcal C (X,Z) $ and $H_c^*$, the theory with compact supports, is associated to a natural subcomplex  $\mathcal C_c (X,Z) $, (see \cite{span} ch.6).

Let us now recall a
few   properties. The first is (see \cite{span} ch.6, p.321, Lemma 11.)
\begin{proposition}\label{numero3}
Let $(X,Z)$   be a pair with $X$   compact  $Z\neq
\emptyset$ closed. Set   $U:=X\setminus Z$. Then there are
natural isomorphism $H_c^q(U)\simeq H^q(X,Z).$
\end{proposition}
In fact this is induced by the      map
of cochains complexes
$\mathcal C_c(U)\to \mathcal C (X,Z)$
composition of the inclusions $\mathcal C_c(U)\to \mathcal C_c (X)\to  \mathcal C (X)$ and of the quotient
$ \mathcal C(X)\to \mathcal C (X,Z) $.

In particular, if we take an  open set  $U$ in a compact space $X$ (for example we could take the one point compactification $U^+$ of  a locally compact space $U$),  we get that
$H^*_c(U)=H^*(X,X\setminus U).$

As an application of this,  assume $Z\subset U$ is closed and   $U$ is open in a compact space $X$. Set $Y=X\setminus U$ and take the triple $(X ,\tilde Z,Y)$  with $\tilde Z=Z\cup Y$.
Consider the commutative    diagram
$$\begin{CD} 0@>>> \mathcal C^*_c(U\setminus Z)@>>>
 \mathcal C_c^* (U)@>>>\mathcal C^*_c(Z)\\ @.@VVV@VVV@VVV\\0@>>> \mathcal C^*(X ,\tilde Z )@>>>
 \mathcal C^* (X,Y)@>>>\mathcal C^*(\tilde Z ,Y)@>>>0 \end{CD}$$
 Using the exactness of the bottom line we deduce the long exact sequence
$$ \begin{CD}\cdots\to H^h_{c}(U\setminus Z)@>i_*>> H^h_{c}(U)@>j^*>> H^h_{c}(Z)@> >>  H^{h+1}_{c}(X\setminus Z)\to\cdots .\end{CD}$$
On the other hand, the top line induces a homomorphism of chain complexes $$\mu: \mathcal C_c^* (U)/ \mathcal C_c^* (U\setminus Z)\to  \mathcal C_c^* (Z)$$ and since the vertical arrows induce isomorphism in cohomology, using the five lemma we easily deduce
\begin{proposition}\label{numero231} The homomorphism $\mu$ induces an isomorphism in cohomology.
\end{proposition}
In order to compare  the Alexander--Spanier and singular cohomology, one needs  to pass to the associated sheaves (see \cite{span} ch.6, p.324).  Thus,  under suitable topological conditions, we obtain a natural isomorphism between Alexander--Spanier and singular cohomology.

In particular   consider   a $C^\infty$-manifold $M$ and  a closed subset $Z\subset M$. Further assume that $Z$ is locally contractible
 (this is the case  for $T^*_GN$ in $T^*N$, as follows from the description of the neighborhood of a $G$-orbit using the slice theorem).
We then have (see \cite{span} ch.6, p.341 Corollary 7) that, under these assumptions,  we can use singular cochains and in fact, in the case of a manifold, singular $C^{\infty}$ cochains to compute cohomology since Alexander Spanier and singular cohomology  are naturally isomorphic in this case.

Integrating on singular $C^\infty$-simplexes we get a commutative     diagram
$$\begin{CD} 0@>>> \mathcal A^*_c(M\setminus Z)@>>>
  \mathcal A_c^* (M) \\  @.@VVV @VVV\\0@>>> _\infty\mathcal C_c^*(M\setminus Z)@>>>
_\infty \mathcal C_c^* (M)  \end{CD}$$
 $\mathcal A^*_c$ being the complex of differential forms with compact supports.
 We deduce a homomorphism  of cochain complexes
 $$\nu: \mathcal A_c^* (M)/ \mathcal A_c^* (M\setminus Z)\to  _\infty\negthinspace\negthinspace\mathcal  C_c^*(M )/_\infty\mathcal  C_c^*(M\setminus Z)$$
 Since the vertical arrows induce isomorphism in cohomology, we  get a de Rham model for $H^*_c(Z)$.
 \begin{proposition}\label{numero43} The homomorphism  $\nu$ induces isomorphism in cohomology. In particular $H^*_c(Z)$ is naturally isomorphic to the cohomology of the complex $ \mathcal A_c^* (M)/ \mathcal A_c^* (M\setminus Z)$.
 \end{proposition}

\subsection{Classifying spaces}

We now take a compact Lie group $G$ and denote by $ B_G $ its
classifying space (which is not locally compact). Recall that $
B_G $ is a polyhedron with finitely many cells in each dimension
and it has a filtration $(B_G)_0\subset \cdots \subset
(B_G)_n\subset (B_G)_{n+1} \subset \cdots \subset  B_G  $ by
compact manifolds with the property that the inclusion
$(B_G)_n\subset  B_G $
 induces isomorphism in  cohomology up to degree $n$.
 For example, if $G$ is a $s$-dimensional torus, $ B_G =\mathbb CP(\infty)^s$ and we may take
$(B_G)_n=\mathbb CP(n)^s$ (indeed in this case the inclusion
induces an isomorphism up to degree $2n-1$).

We denote by $\pi:E_G\to  B_G $ the universal fibration and set
$(E_G)_n=\pi^{-1}((B_G)_n)$. Thus $(E_G)_n$ is also a compact  $C^\infty$
manifold and a principal bundle over $(B_G)_n$.

 Recall now that for any $G$-space $Y$,
$H^*_G(Y)=H^*(Y\times_GE_G)$.

We can define the equivariant cohomology with compact supports
 of a   $G$-space as follows.
 Take $U$ locally compact. Embed  $U$ in his one point compactification $U^+$. The action of $G$ extends to $U^+$ and we   set
\begin{definition} $H^*_{G,c}(U)=H^*_{G}(  U^+,\infty)
.$
\end{definition}
Some remarks are in order.
\begin{itemize}
\item If $U$ is compact, then  $U^+$ is the disjoint union
$U\cup \{\infty\}$ so $H^*_{G,c}(U)=H^*_G(U)$. \item If $U$ is non
compact, then $H^*_{G,c}(U)=H^*(U^+\times_GE_G, B_G )$
where $ B_G =\{\infty\}\times_{G}E_G$.
\item
All the equivariant cohomologies are modules over $H^*_G(pt)$ and
all the homomorphisms are module homomorphisms.
\end{itemize}
Recall that by  the properties of $(B_G)_m$  for any $h\geq 0$,
and for all  $m$ large enough,
$H^r(B_G, R)=H^r((B_G)_m, R)$ for $0\leq r\leq 2h$. So given a
$G$-space $X,$    the spectral sequences of the fibrations
$X\times_GE_G\to B_G$ and $X\times_G(E_G)_m\to (B_G)_m$ have the
same $E^{p,q}_r$ for all $r$ and $p+q\leq h$. In particular we get for any pair $(X,Z)$ of $G$-spaces that for
  large $m$, $H^h_G(X,Z)=H^h(X\times_G(E_G)_m,Z\times_G(E_G)_m)$. From   Proposition \ref{numero3}, we then deduce

\begin{proposition}\label{numero5} Let $X $   be a   $G$-space  with $X$   compact Hausdorff and  $Z\neq \emptyset$ a closed $G$-stable subspace. Set   $U:=X\setminus Z$. Then there is a natural
isomorphism $H_{G,c}^q(  U)\simeq H_G^q(X,Z).$

Furthermore for $m$ large with respect to $h$, $H_{G,c}^h(  U)\simeq H^h_c(U\times_G(E_G)_m).$

\end{proposition}

Take now a $C^{\infty}$ manifold $M$ with a $C^{\infty}$ action of $G$  and a closed $G$-stable subset $Z$ in $M$ which we assume to be locally contractible. For instance if $Z$ is locally triangular as for instance when $Z$ is semi--analytic \cite{Los}. The same is true for $Z\times_G(E_G)_m$ for any $m$ so we can apply  Proposition \ref{numero43} and deduce that
for $m$ large with respect to $h$, $H_{G,c}^h(  Z)$ is the $h$-th cohomology group of the complex $ \mathcal A_c^* (M\times_G(E_G)_m)/ \mathcal A_c^* ((M\setminus Z)\times_G(E_G)_m)$.

But one knows (see \cite{gui-ste99}) that for any $m$   we have a natural morphism of complexes ${\mathcal A}_{G,c}(M)\to \mathcal A_c^* (M\times_G(E_G)_m)$ which induces isomorphism in cohomology in small degree. The same holds also for the open set $M\setminus Z$ so that we get a commutative diagram
$$\begin{CD} 0@>>> {\mathcal A}_{G,c}(M\setminus Z)@>>>
 {\mathcal A}_{G,c}(M) \\  @.@VVV @VVV\\0@>>> \mathcal A_c^*((M\setminus Z)\times_G(E_G)_m))@>>>
\mathcal A_c^* (M\times_G(E_G)_m))  \end{CD}$$
which induces a morphism of complexes
$$\rho:{\mathcal A}_{G,c}(M)/{\mathcal A}_{G,c}(M\setminus Z)\to \mathcal A_c^* (M\times_G(E_G)_m)/ \mathcal A_c^* ((M\setminus Z)\times_G(E_G)_m)$$
From this we immediately deduce

  \begin{proposition}\label{final} $H^*_{G,c}(Z)$ equals the cohomology of the complex ${\mathcal A}^*_{G,c}(Z,M)=
  \mathcal A_c^* (M)/ \mathcal A_c^* (M\setminus Z)$.
 \end{proposition}
\begin{proof} From the above considerations we have,   if  $m$ is large with respect to $h$, $\rho$ induces an isomorphism in cohomology in degree $h$. Since we have seen that   in degree $h$ the cohomology of the complex $ \mathcal A_c^* (M\times_G(E_G)_m)/ \mathcal A_c^* ((M\setminus Z)\times_G(E_G)_m)$ is $H_{G,c}^h(  Z)$, everything follows.
\end{proof}

\end{document}